\documentclass{amsart}
\usepackage{amssymb,amsmath,amsthm,amsfonts}
\usepackage{latexsym,enumerate}
\usepackage{mathrsfs,mathtools}
\usepackage{bbm}
\usepackage{xcolor}

\allowdisplaybreaks
\makeatletter
\def\subsection{\@startsection{subsection}{2}%
  \z@{.5\linespacing\@plus.7\linespacing}
{.5\baselineskip}%
  {\normalfont\centering\scshape}%
}
\makeatother

\newcommand{\nc}{\newcommand}
\nc{\les}{\lesssim}
\nc{\nit}{\noindent}
\nc{\nn}{\nonumber}
\nc{\D}{\partial}
\nc{\diff}[2]{\frac{d #1}{d #2}}
\nc{\diffn}[3]{\frac{d^{#3} #1}{d {#2}^{#3}}}
\nc{\pdiff}[2]{\frac{\partial #1}{\partial #2}}
\nc{\pdiffn}[3]{\frac{\partial^{#3} #1}{\partial{#2}^{#3}}}
\nc{\abs}[1] {\lvert #1 \rvert}
\nc{\cAc}{{\cal A}_c}
\nc{\cE}{{\cal E}}
\nc{\cF}{{\cal F}}
\nc{\cP}{{\cal P}}
\nc{\cV}{{\cal V}}
\nc{\cQ}{{\cal Q}}
\nc{\cGin}{{\cal G}_{\rm in}}
\nc{\cGout}{{\cal G}_{\rm out}}
\nc{\cO}{{\cal O}}
\nc{\Lav}{{\cal L}_{\rm av}}
\nc{\cL}{{\cal L}}
\nc{\cB}{{\cal B}}
\nc{\cZ}{{\cal Z}}
\nc{\cR}{{\cal R}}
\nc{\cT}{{\cal T}}
\nc{\cY}{{\cal Y}}
\nc{\cX}{{\cal X}}
\nc{\cXT}{{{\cal X}(T)}}
\nc{\cBT}{{{\cal B}(T)}}
\nc{\vD}{{\vec \mathcal{D}}}
\nc{\efield}{\mathcal{E}}
\nc{\vE}{{\vec \efield}}
\nc{\vB}{{\vec \mathcal{B}}}
\nc{\vH}{{\vec \mathcal{H}}}
\nc{\mR}{\mathcal R}
\nc{\mF}{\mathcal F}
\nc{\mE}{\mathcal E}
\nc{\ty}{{\tilde y}}
\nc{\tu}{{\tilde u}}
\nc{\tV}{{\tilde V}}
\nc{\Pc}{{\bf P_c}}
\nc{\bx}{{\bf x}}
\nc{\bX}{{\bf X}}
\nc{\bXYZ}{{\bf XYZ}}
\nc{\bY}{{\bf Y}}
\nc{\bF}{{\bf F}}
\nc{\bS}{{\bf S}}
\nc{\dV}{{\delta V}}
\nc{\dE}{{\delta E}}
\nc{\TT}{{\Theta}}
\nc{\dPsi}{{\delta\Psi}}
\nc{\order}{{\cal O}}
\nc{\Rout}{R_{\rm out}}
\nc{\eplus}{e_+}
\nc{\eminus}{e_-}
\nc{\epm}{e_\pm}
\nc{\sgn}{\text{sgn}}
\nc{\eps}{\varepsilon}
\nc{\vnabla}{{\vec\nabla}}
\nc{\G}{\Gamma}
\nc{\w}{\omega}
\nc{\mh}{h}
\nc{\mg}{g}
\nc{\vphi}{\varphi}
\nc{\tlambda}{\tilde\lambda}
\nc{\be}{\begin{equation}}
\nc{\ee}{\end{equation}}
\nc{\ba}{\begin{eqnarray}}
\nc{\ea}{\end{eqnarray}}

\nc{\g}{\gamma}
\nc{\ol}{\overline}

\newtheorem{theorem}{Theorem}[section]
\newtheorem{lemma}[theorem]{Lemma}
\newtheorem{prop}[theorem]{Proposition}
\newtheorem{corollary}[theorem]{Corollary}
\newtheorem{defin}[theorem]{Definition}

\newtheorem{asmp}[theorem]{Assumption}

\nc{\pT}{\partial_T}
\nc{\pz}{\partial_z}
\nc{\pt}{\partial_t}
\nc{\la}{\langle}
\nc{\ra}{\rangle}
\nc{\infint}{\int_{-\infty}^{\infty}}
\nc{\halfwidth}{6.5cm}
\nc{\figwidth}{10cm}
\newcommand{\f}{\frac}

\nc{\nlayers}{L} \nc{\nsectors}{M}
\nc{\indicator}{\mathbf{1}}
\nc{\Rhole}{R_{\rm hole}}
\nc{\Rring}{R_{\rm ring}}
\nc{\neff}{n_{\rm eff}}
\nc{\Frem}{F_{\rm rem}}
\nc{\R}{\mathbb R}
\nc{\mJ}{\mathcal J}
\nc{\C}{\mathbb C}
\nc{\Z}{\mathbb Z}
\nc{\N}{\mathbb N}
\nc{\DD}{\Delta}
\nc{\cD}{\mathcal D}
\nc{\logorm}{\left\|}
\nc{\rnorm}{\right\|}
\nc{\rnormp}{\right\|_{\ell^{p,\eps}}}
\nc{\rar}{\rightarrow}

\begin{document}
	\DeclarePairedDelimiter{\gen}{\langle}{\rangle} 
	\DeclarePairedDelimiter\bra{\langle}{\rvert}
    \DeclarePairedDelimiter\ket{\lvert}{\rangle}
	
	\begin{abstract}
		We consider the higher order Schr\"odinger operator $H=(-\Delta)^m+V(x)$ in $n$ dimensions with real-valued potential $V$ when $n>2m$, $m\in \mathbb N$ when $H$ has a threshold eigenvalue.    We adapt our recent results for $m\geq 1$ when $n>4m$ to lower dimensions $2m<n\leq 4m$  to show that when $H$ has a threshold eigenvalue and no resonances,  the wave operators are bounded on $L^p(\R^n)$ for the natural range $1\leq p<\frac{2n}{n-1}$ when $n$ is odd and $1\leq p<\frac{2n}{n-2}$ when $n$ is even.  We further show that if the zero energy eigenfunctions are orthogonal to $x^\alpha V(x)$ for all $|\alpha|<k_0$, then the wave operators are bounded on $1\leq p<\frac{n}{2m-k_0}$ when $k_0<2m$ in all dimensions $n>2m$.  The range is $p\in [1,\infty)$ and $p\in[1,\infty]$ when $k_0=2m$ and $k_0>2m$ respectively.	The proofs apply in the classical $m=1$ case as well and streamlines existing arguments in the eigenvalue only case, in particular the $L^\infty(\R^n)$ boundedness is new when $n>3$.
	\end{abstract}
	\title[Wave operators for higher order  Schr\"odinger operators]{\textit{$L^p$ boundedness of wave operators for higher order schr\"odinger operators with threshold eigenvalues} } 
	
	\author[M.~B. Erdo\smash{\u{g}}an, W.~R. Green, K. LaMaster]{M. Burak Erdo\smash{\u{g}}an, William~R. Green, Kevin LaMaster}
	\thanks{ The first and third authors were partially supported by the NSF grant  DMS-2154031. The second author is partially supported by Simons Foundation
Grant 511825. }
	\address{Department of Mathematics \\
		University of Illinois \\
		Urbana, IL 61801, U.S.A.}
	\email{berdogan@illinois.edu}
	\address{Department of Mathematics\\
		Rose-Hulman Institute of Technology \\
		Terre Haute, IN 47803, U.S.A.}
	\email{green@rose-hulman.edu}
	\address{Department of Mathematics \\
		University of Illinois \\
		Urbana, IL 61801, U.S.A.}
	\email{kevinl17@illinois.edu}

	\maketitle

\section{Introduction}
We continue the study of wave operators for higher order Schr\"odinger operators related to equations of the form
\begin{align*}
	i\psi_t =(-\Delta)^m\psi +V\psi, \qquad x\in \R^n, \quad  m\in \mathbb N.
\end{align*}
Here $V$ is a real-valued potential with polynomial decay, $|V(x)|\les \la x\ra^{-\beta}$ for some sufficiently large $\beta>0$, and some smoothness conditions, see \cite{EGWaveOp} or Assumption~\ref{asmp:Fourier} below, when $n\geq 4m-1$.  When $m=1$ this is the classical Schr\"odinger equation.  We consider the case  when $(-\Delta)^m+V$ has an eigenvalue at zero energy in dimensions $n>2m$, but no zero energy resonances nor positive eigenvalues.

Denote the free operator by $H_0=(-\Delta)^m$ and the perturbed operator by $H=(-\Delta)^m+V$.   We study the $L^p$ boundedness of the wave operators, which are defined by
$$
W_{\pm}=s\text{\ --}\lim_{t\to \pm \infty} e^{itH}e^{-itH_0}.
$$
For the class  of potentials   we consider, the wave operators exist and are asymptotically complete,  \cite{RS,ScheArb,agmon,Hor,Sche}.   Furthermore,  the intertwining identity
$$
f(H)P_{ac}(H)=W_\pm f((-\Delta)^m)W_{\pm}^*
$$
holds for these potentials where $P_{ac}(H)$ is the projection onto the absolutely continuous spectral subspace of $H$, and $f$ is any Borel function.  The intertwining identity and $L^p$ continuity of the wave operators allows one to obtain $L^p$-based mapping properties of operators of the form $f(H)P_{ac}(H)$ from those of the much simpler operators $f((-\Delta)^m)$.

As usual, we begin with the stationary representation of the wave operators
\begin{align}\label{eqn:stat rep}
	W_+u
	&=u-\frac{1}{2\pi i} \int_{0}^\infty \mR_V^+(\lambda) V [\mR_0^+(\lambda)-\mR_0^-(\lambda)] u \, d\lambda,
\end{align}
where $\mR_V(\lambda)=((-\Delta)^{m}+V-\lambda)^{-1}$, $\mR_0(\lambda)=((-\Delta)^m-\lambda)^{-1}$, and the `+' and `-' denote the usual limiting values as $\lambda$ approaches the positive real line from above and below, \cite{agmon,soffernew}.  As in previous works, \cite{EGWaveOp,EGWaveOpExt,egl}, we consider $W_+$, bounds for $W_-$ follow by conjugation since $W_-=\mathcal C W_+ \mathcal C $, where $\mathcal Cu(x)=\overline{u}(x)$.  Since the identity operator is bounded on all $L^p$ spaces, our main technical results control the contribution of the integral involving the resolvent operators.

In dimensions $n>4m$, there are no resonances at the threshold (zero energy), \cite{soffernew}.  This mirrors the case of dimensions $n>4$ in the classical ($m=1$) Schr\"odinger operator where threshold resonances cannot exist. In the classical case, the existence of threshold eigenvalues limits the upper range of $L^p(\R^n)$ boundedness of the wave operators, generically to $1\leq p<\frac{n}{2}$, \cite{YajNew,GGwaveop,YajNew2}.  With additional orthogonality conditions on the eigenspace this range can be extended to $1\leq p<n$ with the conditions $\int_{\R^n}V(x)\psi(x)=0$ for all eigenfunctions $\psi$ and similarly $1\leq p<\infty$ if additionally $\int_{\R^n}xV(x)\psi(x)=0$ \cite{GGwaveop,YajNew,YajNew2,Yaj3d,Yaj4d}.  Here we prove the analogous result for the higher order Schr\"odinger operators in Corollary~\ref{thm:full} below and capture the $L^\infty$ boundedness in the $m=1$ case, which is new outside of the case when $n=3$ considered by Yajima in \cite{Yaj3d}.

We note that the classical case and dimensions $4m\geq n>2m$, ($n=3,4$), the existence of threshold eigenvalues limits the range of $L^p(\R^n)$ boundedness to $1\leq p<n$ \cite{YajWkp1, GoGr4d2017, JenYaj4d2008} without further orthogonality assumptions \cite{Yaj3d,Yaj4d}. Our proof doesn't distinguish between $m=1$ and $m>1$, and hence applies to the classical case as well where it streamlines the existing arguments in the eigenvalue only case.  We note that the analysis in even dimensions is particularly challenging.

Our main result is to control the low energy portion of the evolution when there is a threshold eigenvalue. Let $\chi$ be  a smooth cut-off function    for a sufficiently small neighborhood of zero, and let $\widetilde \chi=1-\chi$ be the complementary cut-off away from a neighborhood of zero.   As in the papers  \cite{EGWaveOpExt, egl}, it suffices to prove that  for all sufficiently large   $\kappa\geq 0$, the operator 
$$
W_{low,\kappa}u= \frac{1}{2\pi i}  \int_{0}^\infty \chi(\lambda) (\mR_0^+(\lambda) V)^{\kappa} \mR_V^+(\lambda)  (V\mR_0^+(\lambda) )^\kappa V [\mR_0^+(\lambda)-\mR_0^-(\lambda)] u \, d\lambda,
$$
is bounded on $L^p(\R^n)$. Note that the $L^p$ boundedness of the  Born series terms and the high energy part of the remainder $W_{high,\kappa}u$ 
were established in \cite{EGWaveOp} for all $1\leq p\leq \infty$ and $n>2m$ under Assumption~\ref{asmp:Fourier} below.

Throughout the paper, we write $\la x\ra$ to denote $  (1+|x|^2)^{\f12}$, $A\les B$ to say that there exists a constant $C$ with $A\leq C B$, and write $a-:=a-\epsilon$ and $a+:=a+\epsilon$ for some $\epsilon>0$.  We further write $n_\star=n+4$ when $n$ is odd and $n_\star=n+3$ when $n$ is even.  In light of Cheng, Soffer, Wu, and Yao's recent result for low odd dimensions, \cite{CSWY25}, we first state the following result for all dimensions $2m<n\leq 4m$, both even and odd, when there is a zero energy eigenvalue but no resonances, without further orthogonality assumptions.

\begin{theorem}\label{thm:low d natural}
	Let $m\geq 1$ be fixed and $2m<n\leq 4m$.  Assume that $|V(x)|\les \la x\ra^{-\beta}$, where $V$ is a real-valued potential on $\R^n$ and $\beta>\max(8m-n,n_\star)+2 $. Let $H=(-\Delta)^m+V(x)$ have an eigenvalue at zero, but no positive eigenvalues.  Then, for all sufficiently large $\kappa$, $W_{low,\kappa}$ extends to a bounded operator on $L^p(\R^n)$ for all $1\leq p< \frac{2n}{n-1}$ when $n$ is odd and $1\leq p <\frac{2n}{n-2}$ when $n$ is even.	 
\end{theorem}
We note that this has a slightly less stringent assumption than \cite{CSWY25}, which requires $\beta>n+4k+5=8m-n+11$ for $n$ odd.
In our previous work, \cite{egl}, we showed that when $n>4m$ the wave operators extend to bounded operators on $L^p(\R^n)$ for all $1\leq p<\frac{n}{2m}$ without further assumptions on the eigenspace.  In dimensions $2m<n\leq 4m$, the eigenspace is naturally orthogonal to $x^{\alpha}V(x)$ for all $\alpha$ with $|\alpha|\leq 2m-\lceil\frac{n}{2}\rceil$.  We show that, similar to known results in the $m=1$ case, that if the zero energy eigenspace has additional orthogonality, one can extend the range of $p$ upward.  
Specifically, we show that
\begin{theorem}\label{thm:main_low}
	Let $n>2m\geq 2$ and $\max(1,2m-\lceil\frac{n}{2}\rceil+1) \leq k_0\leq 2m+1$. Assume that $|V(x)|\les \la x\ra^{-\beta}$, where $V$ is a real-valued potential on $\R^n$ and $ \beta>\max(8m-n,n_\star)+2$.  Assume that $H=(-\Delta)^m+V(x)$ have an eigenvalue at zero with an eigenspace orthogonal to $\{x^\alpha V(x):|\alpha|\leq k_0-1\}$ and no positive eigenvalues. Then for all sufficiently large $\kappa$,   $W_{low,\kappa}$ extends to a bounded operator on $L^p(\R^n)$ for all $1\leq p<\frac{n}{2m-k_0}$ in the case $k_0<2m$. If $k_0=2m$, then the range is $1\leq p<\infty$, and the range is $1\leq p\leq \infty$ if $k_0=2m+1$.
\end{theorem}
In dimensions $2m<n<4m$, the claim holds for all $\kappa\geq 0$.
Also note that Theorem~\ref{thm:low d natural} is a special case of Theorem~\ref{thm:main_low} by taking $k_0=\max(1,2m-\lceil\frac{n}{2}\rceil+1)=2m-\lceil\frac{n}{2}\rceil+1$, when $2m<n\leq 4m$.

We note that the $p=\infty$ endpoint result is new even in the  classical ($m=1$) case in dimensions $n>3$.  In these cases, it was shown that the wave operators are bounded when $1\leq p<\infty$ if the eigenspace is orthogonal to both $V(x)$ and $xV(x)$, \cite{GGwaveop,YajNew2,Yaj4d}.  We note that Yajima proved the $p=\infty$ boundedness when $n=3$ in \cite{Yaj3d}.  When $m=1$ and $n=4$, we only consider the case when $k_0=3$ in our arguments since $k_0=1,2$ were considered by Yajima in \cite{Yaj4d}.

We need sufficiently large $\kappa$ when $n\geq 4m$ due to local singularities of the free resolvents that are not  square integrable.  The case of $n=4m$ has technical challenges in both the high and low energy regimes due to the existence of zero energy resonances and the need for smoothness of $V$ in the high energy regime, \cite{EGG}.  Our arguments are fairly streamlined, we avoid long operator-valued expansions of the perturbed resolvent by adapting the methods in \cite{EGWaveOp,EGWaveOpExt} to control the singularity as $\lambda \to 0$ that occurs when there is a zero energy eigenvalue.  We also treat both even and odd dimensions in a unified manner.

We recall  the first $L^p$ boundedness result in the seminal paper of Yajima, \cite{YajWkp1}, for $m=1$, $n\geq 3$ and $1\leq p\leq \infty$ for small potentials.  For large potentials, the main difficulty is in controlling the contribution of $W_{low, \kappa}$.  The behavior of this operator differs in even and odd dimensions. In \cite{YajWkp1,YajWkp2,YajWkp3}, Yajima removed smallness or positivity assumptions on the potential for all dimensions $n\geq 3$.  These arguments were simplified and Yajima further considered the effect of zero energy eigenvalues and/or resonances in \cite{Yaj} when $n$ is odd and with Finco in \cite{FY} when $n$ is even for $n>4$ to establish boundedness of the wave operators when $\frac{n}{n-2}<p<\frac{n}{2}$.  These results were further extended to show that the range of $p$ is generically $1\leq p<\frac{n}{2}$ in the presence of a zero energy eigenvalue, and that the upper range of $p$ may be larger under certain orthogonality conditions, \cite{YajNew,GGwaveop,YajNew2}, which was proved independently by Yajima and separately by Goldberg and the second author.

For completeness, we pair this with existing high energy bounds for the wave operators from \cite{EGWaveOp}, which require the following assumptions. 
Let  $\mF({f})$ denote  the Fourier transform of $f$.
\begin{asmp}\label{asmp:Fourier}
	For some $0<\delta \ll 1$, assume that the real-valued potential $V$ satisfies the condition
	\begin{enumerate}[i)]
		
		\item $\big\| \la \cdot \ra^{\frac{4m+1-n}{2}+\delta} V(\cdot)\big\|_{2}<\infty$   when $2m<n<4m-1$,
		
		\item $\big\|\la \cdot \ra^{1+\delta}V(\cdot)\big\|_{H^{\delta}}<\infty$ when $n=4m-1$,
		
		\item  $\big\|\mathcal F(\la \cdot \ra^{\sigma} V(\cdot))\big\|_{L^{ \frac{n-1-\delta}{n-2m-\delta} }}<\infty$ for some $\sigma>\frac{2n-4m}{n-1-\delta}+\delta$   when $n>4m-1$.
	\end{enumerate}
	
\end{asmp}
In \cite{EGWaveOp}, by adapting Yajima's $m=1$ argument in \cite{YajWkp1}, the first two authors showed that the contribution of the terms of the Born series are bounded under these assumptions.
In addition, it was shown that if $|V(x)|\les \la x\ra^{-\beta}$ for some $\beta>n+5$ when $n$ is odd and $\beta>n+4$ when $n$ is even and if $\kappa$ is sufficiently large (depending on $m$ and $n$), then $W_{high,\kappa}$ is a bounded operator on $L^p$ for all $1\leq p\leq \infty$ provided there are no positive eigenvalues, \cite{EGWaveOp,EGWaveOpExt}.  The lack of positive eigenvalues is a vital assumption for higher order operators since there may be positive eigenvalues even for smooth, compactly supported potentials, see \cite{soffernew}.

Combining these facts with Theorem~\ref{thm:main_low}, we have the following result.
\begin{corollary}\label{thm:full} 
	Fix  $n>2m\geq 2$, $k_0\geq 1$, and let  $V$  satisfy Assumption~\ref{asmp:Fourier}.  Under the hypotheses of Theorem~\ref{thm:main_low},
	the wave operators extend to bounded operators on $L^p(\R^n)$ for all $1\leq p< \frac{n}{2m-k_0}$ when $k_0<2m$.  If $k_0= 2m$, then this range is $1\leq p<\infty$, and if $k_0>2m$ the wave operators are also bounded on $L^\infty(\R^n)$.
\end{corollary}
By applying the intertwining identity and the known $L^p\to L^{p'}$ dispersive bounds when $p'$ is the H\"older conjugate of $p$ for the free solution operator $e^{-it(-\Delta)^m}$, we obtain the corollary below. 
\begin{corollary}
	Under the assumptions of Corollary~\ref{thm:full} with $2\leq p\leq\infty$ in the appropriate ranges, we obtain the dispersive estimates
	$ \|e^{-itH}P_{ac}(H)f\|_{p} \les |t|^{-\frac{n}{ m}(\frac{1}{2}-\frac{1}{p}) } \|f\|_{p'}. $ 
\end{corollary}

The study of the $L^p$ boundedness of the wave operators in the higher order $m>1$ case was initiated fare more recently than that of the classical case $m=1$.  The first paper to study the $m>1$ problem was that of Goldberg and the second author, \cite{GG4wave}, in which they showed that the wave operators are bounded on $1<p<\infty$  $m=2$ and $n=3$.  The case $n>2m$ was studied by the first two authors in \cite{EGWaveOp,EGWaveOpExt}.  In \cite{MWY},  Mizutani, Wan, and Yao  considered the case of $m=2$ and $n=1$ showing that the wave operators are bounded when $1<p<\infty$, but not when $p=1,\infty$, where weaker estimates involving the Hardy space or BMO were proven depending on the type of threshold obstruction.

In \cite{EGG} the first two authors and Goldberg showed that a certain amount of smoothness of the potentials is necessary to control the large energy behavior in the $L^p$ boundedness.  The effect of a zero energy eigenvalue without additional orthogonality assumptions in high dimensions $n>4m$ was studied by the authors in \cite{egl}.  Mizutani, Wan, and Yao \cite{MWY3d,MWY3d2} studied the endpoint behavior and the effect of zero energy resonances when $m=2$ and $n=3$, showing that the wave operators are unbounded at the endpoints $p=1,\infty$, as well as study the effect of zero energy obstructions.

More recently in \cite{GY}, Galtbayar and Yajima considered the case  $m=2$ and $n=4$ showing that the wave operators are bounded on $1<p<\infty$ if zero is regular with restrictions on the upper range of $p$ if zero is not regular depending on the type of resonance at zero.   Cheng, Soffer, Wu, and Yao \cite{CSWY25} studied the general threshold obstruction case for odd $n$ with $1\leq n\leq 4m-1$ with sharp results in the presence of threshold resonances.  This recent work on higher order, $m>1$, Schr\"odinger operators has roots in the work of Feng, Soffer, Wu and Yao \cite{soffernew} which considered time decay estimates between weighted $L^2$ spaces.

The paper is organized as follows.  In Section~\ref{sec:res exp}, we develop a representation of the $W_{low,\kappa}$ in terms of a sum of a bounded and sufficiently smooth operator and a singular but finite rank operator acting on free resolvents, see \eqref{Gammakrep} below. We  also state  Proposition~\ref{prop:low tail eigenspace}, which we use to control  the singular part of $\Gamma_\kappa$. We then reduce the proof of Theorem~\ref{thm:main_low} to the representation \eqref{Gammakrep} and Proposition~\ref{prop:low tail eigenspace}. In Section~\ref{sec:prop pf}, we prove Proposition~\ref{prop:low tail eigenspace}  partially relying on    the arguments in \cite{EGWaveOpExt,egl} and on resolvent expansions tailored to the case of zero energy eigenvalues with a general level of orthogonality.  The remaining sections of the paper are devoted to the proof of representation \eqref{Gammakrep}.

\section{Preliminaries}\label{sec:res exp}
In this section we provide the main technical results needed to prove Theorem~\ref{thm:main_low}, of which Theorem~\ref{thm:low d natural} is a special case.  The main technical achievement is Proposition~\ref{prop:low tail eigenspace}, which is developed in some generality and may be used to prove boundedness of the wave operators in the presence of eigenvalues and/or resonances.

It is convenient to use a change of variables to represent $W_{low,\kappa}$ as  
$$
\frac{m}{\pi i}\int_{0}^\infty \chi(\lambda ) \lambda^{2m-1} (\mR_0^+(\lambda^{2m}) V)^{\kappa } \mR_V^+(\lambda^{2m})  (V\mR_0^+(\lambda^{2m}) )^\kappa V   [\mR_0^+(\lambda^{2m})-\mR_0^-(\lambda^{2m})]  \, d\lambda
$$
For simplicity, we write $\chi(\lambda)$ to mean $\chi(\lambda^{2m})$.
We begin by using the symmetric resolvent identity: $$
\mR_V^+(\lambda^{2m})V=\mR_0^+(\lambda^{2m})vM^+(\lambda)^{-1}v.
$$
Here, $M^+(\lambda)=U+v\mR_0^+(\lambda^{2m})v$, where   $v=|V|^{\f12}$, $U(x)=1$ if $V(x)\geq 0$ and $U(x)=-1$ if $V(x)<0$.  Throughout, we'll assume that $H$ has no embedded eigenvalues and only an eigenvalue at zero. In which case, we have suitable expansions for $M^+(\lambda)^{-1}$ on $(0,\lambda_0)$ for some sufficiently small $\lambda_0 <1$. We select the cut-off $\chi$ to be supported in $(-\lambda_0,\lambda_0)$.   
Using the symmetric resolvent identity,  we have
\be \label{Wlowk}
W_{low,\kappa}=	\frac{m}{\pi i}\int_{0}^\infty \chi(\lambda) \lambda^{2m-1} \mR_0^+(\lambda^{2m}) v\Gamma_\kappa(\lambda) v   [\mR_0^+(\lambda^{2m})-\mR_0^-(\lambda^{2m})]  \, d\lambda,
\ee
where $\Gamma_0(\lambda):= M^+(\lambda)^{-1} $ and for $\kappa\geq 1$
\be \label{Gammalambda}
\Gamma_\kappa(\lambda):= Uv\mR_0^+(\lambda^{2m}) \big(V\mR_0^+(\lambda^{2m})\big)^{\kappa-1} v[M^+(\lambda)^{-1}]v \big( \mR_0^+(\lambda^{2m})V\big)^{\kappa-1} \mR_0^+(\lambda^{2m})vU.
\ee

We define the following subspaces and their associated projections as in \cite{soffernew}, also see \cite{egl}:   
\begin{align}
	S_1L^2(\R^n):=&\ \mathrm{Ker}(T_0)\nonumber\\
	S_{i+1}L^2(\R^n):=& \ \{\phi \in S_1L^2:\gen{v\phi, x^\alpha}=0, |\alpha|\leq i-1\},&i\geq 1,\label{defn:Si}
\end{align}
where $T_0=M^+(0)=U+v\mR_0^+(0)v$. 

Note  that when $n>4m$, a distributional solution of $H\psi =0$ is a zero energy eigenfunction of $H$ if and only if  $\phi:=Uv\psi\in  S_1L^2 $. However, for $2m<n\leq 4m$, the space $S_1L^2$ corresponds to eigenvalues and various types of resonances, classified by their membership to $S_{i+1}L^2$, $0\leq i\leq k$, where 
$$k=k(n,m):=2m-\lceil\tfrac{n}{2}\rceil+1,$$ 
see \cite{soffernew}. Since we are considering the eigenvalue only case, in dimensions $2m<n\leq 4m$  we have $S_1=S_{k+1}$.
For this reason, in Theorem~\ref{thm:main_low}, we take $k_0\geq k$ in dimensions $2m<n\leq 4m$. We take $k_0\geq 1$ in higher dimensions, as we already studied the case $k_0=0$  in \cite{egl}.

We'll split our operator $\Gamma_\kappa(\lambda)$  into an operator containing the singularities in the spectral parameter, $\lambda$,  and a bounded operator which is sufficiently smooth in $\lambda$, quantified by  Proposition~2.1 of \cite{EGWaveOpExt}:
\begin{prop}\label{prop:low tail low d} Fix  $n>2m\geq 2$  and let $\Gamma$ be a $\lambda $ dependent absolutely bounded operator. Let 
	$$
	\widetilde \Gamma (x,y):= \sup_{0<\lambda <\lambda_0}\Big[|\Gamma(\lambda)(x,y)|+ \sup_{1\leq \ell\leq  \lceil \tfrac{n}2\rceil +1  } \big|\lambda^{\ell-1} \partial_\lambda^\ell \Gamma(\lambda)(x,y)\big| \Big].
	$$
	For $2m<n<4m$ assume that $\widetilde \Gamma $ is bounded on $L^2$, and for $n\geq 4m$  assume that $\widetilde\Gamma $ satisfies   \be\label{eq:tildegamma}
	\widetilde \Gamma (x,y)  \les \la x\ra^{-\frac{n}2-}\la y\ra^{-\frac{n}2-}.
	\ee Then the operator with kernel 
	\be\label{Kdef}
	K(x,y)=\int_0^\infty \chi(\lambda) \lambda^{2m-1} \big[\mR_0^+(\lambda^{2m}) v \Gamma(\lambda)  v [\mR_0^+(\lambda^{2m}) -\mR_0^-(\lambda^{2m})]\big](x,y)   d\lambda 
	\ee
	is bounded on $L^p$ for $1\leq p\leq \infty$ provided that $\beta>n$.
\end{prop}
To define the  singular part, we need to develop some notation.  We say a finite rank operator is of the form $(k_1,k_2,\alpha)$ if it can be written as
\begin{align*}
	\omega^{\alpha}(\lambda)S_{k_1+1}\Gamma S_{k_2+1}, \quad \text{when }k_1,k_2>0,\\
	\omega^{\alpha}(\lambda)S_{k_1+1} \Gamma , \quad \text{when }k_2=0\\
	\omega^{\alpha}(\lambda)\Gamma S_{k_2+1} , \quad \text{when }k_1=0,
\end{align*}
where, in each occurance, $\Gamma$ is a $\lambda$ independent,  bounded operator, and $\omega^\alpha(\lambda)$ is an arbitrary function of $\lambda$ satisfying
\be\label{omegalambda} | \partial^N_\lambda \omega^\alpha(\lambda)|\les \lambda^{\alpha-N},\,\,N=0,1,2,...,  \lceil \tfrac{n}2\rceil +1. 
\ee

We will prove that under the assumptions of Theorem~\ref{thm:main_low}, we have
\be\label{Gammakrep} 
\Gamma_\kappa(\lambda)=\lambda^{-2m} \Gamma_s (\lambda) +\Gamma_e(\lambda),
\ee
where the singular part, $\Gamma_s$, is a finite linear combination of operators of the form 
$$(k_1,k_2,\alpha)\in  \{(k_0,k_0,0), (k_0,0,k_0), (0,k_0,k_0)\}, \,\,\text{ when } 2m< n<4m,$$
$$
(k_1,k_2,\alpha)=(k_0,k_0,0),\,\,\text{ when } n\geq 4m, 
$$  
and the error term, $\Gamma_e(\lambda)$, satisfies the hypothesis of Proposition~\ref{prop:low tail low d}. Therefore, the contribution of the error term  to 
$W_{low,\kappa} $, see \eqref{Wlowk}, is bounded on $L^p$ for all $1\leq p\leq \infty$. 

To take care of the contribution of the singular part, we will prove  
\begin{prop}\label{prop:low tail eigenspace} 
	Fix $n>2m\geq 2$,   $\alpha \geq 0$, $k_1\geq 0$, and $0\leq k_2 \leq n $ so that 
	\be\label{k1k2alpha}
	k_1+k_2+\alpha > 4m-n\,\,\text{ and }\,\, k_2+\alpha>0.
	\ee
	In addition, assume that $k_1\geq 1$ for $n\geq 4m$. Let $\phi_1\in S_{k_1+1}L^2$ if $k_1\geq 1$, and $\phi_2\in S_{k_2+1}L^2$ if $k_2\geq 1$, otherwise we assume $\phi_j\in L^2$.  If $\beta>\max(n,2m+k_1,2m+k_2)$, then the operator with kernel 
	\be\label{Kdef_eigen_2}
	K(x,y)=\int_0^1 \omega^{\alpha-1}(\lambda) \big[\mR_0^+(\lambda^{2m}) v\phi_1\overline{\phi_2}v[\mR_0^+(\lambda^{2m}) -\mR_0^-(\lambda^{2m})]\big](x,y)   d\lambda 
	\ee
	is bounded on $L^p$ for $1\leq p<\frac{n}{2m-k_2-\alpha}$. Moreover, the range is $[1,\infty)$ if $k_2+\alpha=2m$ and it is $[1,\infty]$ if $k_2+\alpha>2m$.
\end{prop}
This finishes the proof of Theorem~\ref{thm:main_low} noting that for each $(k_1,k_2,\alpha)$ occurring in the representation \eqref{Gammakrep} we have $\alpha+k_2\geq k_0\geq 1$ and 
$$k_1+k_2+\alpha=2k_0\geq \max(2,2k)=\max(2,4m-2\lceil\tfrac{n}2\rceil+2)>4m-n.$$

With  representation \eqref{Gammakrep}, we reduce the effect of the zero energy obstruction to fairly explicit operators that we control using Proposition~\ref{prop:low tail eigenspace}.  The added layers of orthogonality, measured by $k_1,k_2$, allows us to push the upper bound on $p$ upward.  Although, in the analysis considered here, we take $(k_1,k_2,\alpha)\in \{(k_0,k_0,0), (k_0,0,k_0), (0,k_0,k_0)\}$, we developed these tools with more flexibility to handle the effects of resonances in lower dimensions. We note that Proposition~\ref{prop:low tail eigenspace} can be used to prove results when there are resonances in dimensions $2m<n\leq 4m$.
However, in light of Chen, Soffer, Wu, and Yao's recent result for low odd dimensions, \cite{CSWY25}, we don't pursue that in this paper.

To prove these results  we need the following representations of the free resolvent from Lemma 2.3, Lemma 2.4, and Remark 2.5 in \cite{EGWaveOpExt}, that have their roots in \cite{EGWaveOp}.
\begin{lemma}\label{prop:F} 
	Let $n>2m\geq 2$. Then,   we have the representations, with $r=|x-y|$, 
	$$
	\mR_0^+(\lambda^{2m})(x,y)
	=  \frac{e^{i\lambda r}}{r^{n-2m}} F(\lambda r ), \text{ and}$$ 
	\be\label{Frep}
	[\mR_0^+(\lambda^{2m})-\mR_0^-(\lambda^{2m})](x,y)= \lambda^{n-2m}  \big[ e^{i\lambda r}F_+(\lambda r)+e^{-i\lambda r}F_-(  \lambda r)\big],
	\ee
	where, for all $N\geq 0$, 
	\be\label{Fbounds}
	|\partial_\lambda^N F(\lambda r)|\les \lambda^{-N} \la  \lambda r \ra^{\frac{n+1}2 -2m},\, \,\,\, \,\,\, 
	|\partial_\lambda^N F_\pm(\lambda r)|\les \lambda^{-N} \la  \lambda r\ra^{\frac{1-n}2}.
	\ee 
\end{lemma}

We utilize the following notation to streamline the low energy arguments. 
Analogous to the definition of  $\omega^\alpha(\lambda)$ abov;
for given $\alpha\in\R$, we denote any function of $\lambda, r\in \R^+$  by $\omega^\alpha(\lambda,r)$ provided that the function is 0 for $\lambda \geq \lambda_0$ and it satisfies  
$$| \partial^N_\lambda \omega^\alpha(\lambda,r)|\les \lambda^{-N} \la \lambda r\ra^\alpha,\,\,N=0,1,2,...$$
The function may differ in each occurrence, for example, for $\lambda<\lambda_0$
$$
[\mR_0^+(\lambda^{2m})-\mR_0^-(\lambda^{2m})](x,y)= \lambda^{n-2m}  \big[ e^{i\lambda r}\omega^{\frac{1-n}2}(\lambda,r) +e^{-i\lambda r}\omega^{\frac{1-n}2}(\lambda,r)\big],\quad  r=|x-y|.
$$
In some cases, the function may depend on other variables. Then,  it should be understood that the bound above is uniform in all other variables. Note that for all $\alpha_1,\alpha_2\in\R$,
$$
\la \lambda r\ra^{\alpha_1} \omega^{\alpha_2}(\lambda,r)  =\omega^{\alpha_1+\alpha_2}(\lambda,r), \,\,\,\text{and}\,\,\,\omega^{\alpha_1} \omega^{\alpha_2} =\omega^{\alpha_1+\alpha_2}.
$$

\section{Proof of Proposition \ref{prop:low tail eigenspace}}\label{sec:prop pf}
In this section we prove Proposition~\ref{prop:low tail eigenspace}.  To do so, we prove a series of Lemmas that utilize the orthogonality between the eigenspace and the potential that allow us to push the range of $p$ upward.
We start with a standard lemma providing decay bounds for the zero energy eigenfunctions (and resonances) of $H$ and associated $\phi\in S_1L^2(\R^n)$:
\begin{lemma}\label{lem:eigen_decay}
	If $|V(x)|\lesssim\gen{x}^{-\beta}$ for some $\beta> 2m$  and $\phi\in S_1L^2$,  then $|\phi(x)| \lesssim\gen{x}^{2m-n-\frac{\beta}2}$.
\end{lemma}

\begin{proof}
	Recall that $\phi=Uv\psi$, where $\psi\in L^\infty$ is a distributional solution of $H\psi=0$, see, e.g., Section 8 of \cite{soffernew}. Note that the boundedness of $\psi$ may be seen  by iterating the identity $\psi=-G_0^0V\psi=(-G_0^0V)^k\psi$ for sufficiently large $k$. Here $G_0^0=(-\Delta^m)^{-1}$ is the operator with kernel $G_0^0(x,y)=a_0|x-y|^{2m-n}$ for some constant $a_0\in \R\setminus\{0\}$.
	Therefore, we immediately have $|\phi|=|v\psi|\lesssim\gen{x}^{-\beta/2}$.  Using the identity $\phi=-UvG_0v\phi$ and properties of $I_{2m}$, if we assume $|\phi|\lesssim\gen{x}^{-\alpha}$ for some $n>\alpha+\frac{\beta}{2}> 2m$, then
	\[|\phi(x)|=c\left|\int_{\R^n}\frac{v(x)v(y)\phi(y)}{|x-y|^{n-2m}}\ \!dy\right|\lesssim\gen{x}^{-\f\beta2}\int_{\R^n}\gen{y}^{-\f\beta2-\alpha}|x-y|^{2m-n}\ \!dy\lesssim\gen{x}^{-\beta-\alpha+2m}.\]
	Noting that $\beta+\alpha-2m>\alpha$, this process may be iterated to achieve the desired decay.
	If $\alpha+\frac{\beta}{2}\geq n$ then similarly
	\[|\phi(x)|=c\left|\int_{\R^n}\frac{v(x)v(y)\phi(y)}{|x-y|^{n-2m}}\ \!dy\right|\lesssim\gen{x}^{-\f\beta2}\int_{\R^n}\gen{y}^{-\f\beta2-\alpha}|x-y|^{2m-n}\ \!dy\lesssim\gen{x}^{-\f\beta2+2m-n}.\]
\end{proof}
Next we provide two lemmas utilizing the cancellation of $S_{k_1+1}$, $S_{k_2+1}$ with $\mR_0^+(\lambda^{2m})v$ and $v[\mR_0^+-\mR_0^-](\lambda^{2m})$ respectively, that are vital to proving Proposition~\ref{prop:low tail eigenspace}.  In both lemmas we utilize the integral kernels for the free resolvents in \eqref{Kdef_eigen_2}, leading to further integrals in the spatial variables $z_1$ and $z_2$ respectively.

Inspired by Galtbayar and Yajima's recent works \cite{AY4d,AY2dhi}, we have the following.   
\begin{lemma}\label{lem:right_side_cancel}
	
	If $\phi\in S_{k_2+1}L^2$ for some $1\leq k_2\leq n$,  then if $|V(x)|\les \la x\ra^{-\beta}$ for some $\beta>2m+k_2$, and any $0<\varepsilon<k_2$, we have $\varphi(x):=|D|^{-k_2+\varepsilon}[v\phi](x)\lesssim\gen{x}^{-n-\varepsilon}$.  Further, for any Borel function $h$, we have the identity
	\begin{multline*}
		\int^\infty_0 h(\lambda) \int_{\R^n}[\mR_0^+(\lambda^{2m})-\mR_0^-(\lambda^{2m})](r_2)[v\phi](z_2)\ \!dz_2 \ \!d\lambda\\ =\int^\infty_0h(\lambda)   \lambda^{k_2-\varepsilon}\int_{\R^n} [\mR_0^+(\lambda^{2m})-\mR_0^-(\lambda^{2m})](r_2)\varphi (z_2)\ \!dz_2\ \!d\lambda.
	\end{multline*}
\end{lemma}
\begin{proof}[Proof of Lemma~\ref{lem:right_side_cancel}]
	Viewing the resolvent difference as the spectral projection, via the Stone's formula one has
	\[\int^\infty_0 h(\lambda) [\mR_0^+(\lambda^{2m})-\mR_0^-(\lambda^{2m})]v\phi\ \!d\lambda=\int^\infty_0 h(\lambda) \lambda^{k_2-\varepsilon}[\mR_0^+(\lambda^{2m})-\mR_0^-(\lambda^{2m})]|D|^{-k_2+\varepsilon}v\phi\ \!d\lambda.\]
	See also Lemma~3.1 in \cite{AY2dhi}.
	
	By Lemma~\ref{lem:eigen_decay}, we have $|v\phi(y)|\les \la y\ra^{2m-n-\beta}$.  
	Therefore,  
	$$  \varphi(x) = \  \int_{\R^n}\frac{v\phi(y)}{|x-y|^{n-k_2+\varepsilon}} dy $$ is a bounded function of $x$ since $n-k_2+\varepsilon+n+\beta-2m>n$, which suffices when $|x|\lesssim 1$.

	When $|x|\gtrsim 1$, $|y|\ll |x|$, we have the expansion
	\begin{align}\label{eqn:norm_diff_exp}
		\frac{|x|^{\alpha}}{|x-y|^{\alpha}}= & \left(1+\frac{|y|^2}{|x|^2}-\frac{2x\cdot y}{|x|^2}\right)^{-\frac{\alpha}2} =   \sum_{j=0}^k c_j  \left(\frac{|y|^2}{|x|^2}-\frac{2x\cdot y}{|x|^2}\right)^j+O(|y|^{k+1}|x|^{-k-1}) \\
		=&\sum_{\substack{2j_1+j_2\leq k\\j_1,j_2\geq 0}}c_{j_1,j_2}\frac{|y|^{2j_1}(x\cdot y)^{j_2}}{|x|^{2j_1+2j_2}}+\sum_{\substack{k<2j_1+j_2\\j_1,j_2\geq 0, j_1+j_2\leq k}}c_{j_1,j_2}\frac{|y|^{2j_1}(x\cdot y)^{j_2}}{|x|^{2j_1+2j_2}}+O(|y|^{k+1}|x|^{-k-1}).\nonumber\\
		=&\sum_{\substack{2j_1+j_2\leq k\\j_1,j_2\geq 0}}c_{j_1,j_2}\frac{|y|^{2j_1}(x\cdot y)^{j_2}}{|x|^{2j_1+2j_2}} +O(|y|^{k+1}|x|^{-k-1})=:Q_k(x,y)+O(|y|^{k+1}|x|^{-k-1}).\nonumber
	\end{align}
	Note that $Q_k(x,y)$ is a polynomial in $y$ of degree $\leq k$ and $|Q_k(x,y)|\les |y|^k|x|^{-k}\les  |y|^{k+1}|x|^{-k-1}$ when $|y|\gtrsim |x|$. Therefore, for all $|x|\gtrsim 1$ and $y\in \R^n$, we have (with $P_k(x,y):=|x|^{-\alpha}Q_k(x,y)$)
	$$
	\Big|\frac{1}{|x-y|^{\alpha}}-P_k(x,y)\Big|\les \frac{|y|^{k+1}}{|x|^{\alpha+k+1}}+\frac{\chi_{|y|\gtrsim |x|}}{|x-y|^{\alpha}}.
	$$
	Using this with $k=k_2-1$ and $\alpha=n-k_2+\varepsilon$, we have  for $|x|\gtrsim 1$
	\begin{multline*}
		|\varphi(x)|=\  \Big|\int_{\R^n}\frac{v\phi(y)}{|x-y|^{n-k_2+\varepsilon}}\ \!dy\Big| 
		=\  \Big|\int_{\R^n}\left(\frac{1}{|x-y|^{n-k_2+\varepsilon}}- P_{k_2-1}(x,y) \right)v\phi(y)\ \!dy\Big|\\
		\lesssim\  \int_{\R^n}\Big( \frac{|y|^{k_2}}{|x|^{n+\varepsilon} } + \frac{\chi_{|y|\gtrsim |x|}}{|x-y|^{n-k_2+\varepsilon}}\Big) \gen{y}^{-\beta+2m-n}\ \!dy \\
		\les \la x\ra^{-n-\varepsilon} \int_{\R^{n}}\bigg( \la y\ra^{-\beta+2m-n+k_2}+\frac{\la y\ra^{2m-\beta-\varepsilon}}{|x-y|^{n-k_2+\varepsilon}}\bigg)\, dy
		\lesssim\  |x|^{-n-\varepsilon},
	\end{multline*}
	provided that that $\beta>2m+k_2$ and $n\geq k_2$. 
\end{proof}

The lemma above will suffice to control the integral kernel \eqref{Kdef_eigen_2} in Proposition~\ref{Kdef_eigen_2} in the cases when  $|x| < 1$ and when $1< |x| \les |z_1|$ without a need to utilize the orthogonality relations for $\phi_1\in S_{k_1+1}$ on the left. For the case $1< |x|$ and $|z_1|\ll |x|$, we need to utilize the orthogonality of the eigenspace also on the left hand side, which is encapsulated in the following lemma:
\begin{lemma}\label{lem:left_side_cancel}
	If $\phi\in S_{k_1+1}L^2$ for some $k_1\geq 1$, then 
	\begin{multline*}\int_{\R^n}\chi_{|z_1|\ll |x|} \mR_0^+(\lambda^{2m})(r_1)v\phi(z_1)\ \!dz_1=
		\sum_{j=0}^{k_1-1}\int_{\R^n} e^{i\lambda |x|} \frac{\chi_{|z_1|\gtrsim  |x|}  v\phi(z_1) |z_1|^j}{|x|^{n-2m+j}}  \omega^{\frac{n+1}{2}-2m+j}(\lambda,|x|)  \ \!dz_1
		\\+\int^1_0\frac{(1-s)^{k_1-1}}{(k_1-1)!}\int_{\R^n}e^{i\lambda r_1}\frac{\chi_{|z_1|\ll |x|} v\phi(z_1)  |z_1|^{k_1} }{r_1^{n-2m+k_1}} \omega^{\frac{n+1}{2}-2m+k_1}(\lambda,r_1) \ \!dz_1ds
	\end{multline*}
	where $r_1=|x-sz_1|$.
\end{lemma}
We note that the cut-offs are different in the leading $k_1$ terms in the first finite sum, than in the final summand.  
\begin{proof}[Proof of Lemma~\ref{lem:left_side_cancel}]
	Denoting $f(s):=\mR_0^+(\lambda^{2m})(|x-sz_1|)$, utilizing Taylor's formula with  respect to $s$ yields
	\begin{multline}\label{eqn:mR_0expn}
		\mR_0^+(\lambda^{2m})(|x-z_1|)=f(1)\\
		=f(0)+\partial_sf(0)+\ldots+\frac{1}{(k_1-1)!}\partial_s^{k_1-1}f(0)+\int^1_0\frac{(1-s)^{k_1-1}}{(k_1-1)!}\partial_s^{k_1}\mR_0^+(\lambda^{2m})(|x-sz_1|)\ \!ds
	\end{multline}
	
	We observe that $\partial^i_s|x-sz_1|$ is a linear combination of terms of the form
	\begin{align}\label{eqn:partials in s ugly}
		|x-sz_1|^{1-2m_1-2m_2}((x-sz_1)\cdot z_1)^{m_1}|z_1|^{2m_2}
	\end{align}
	where $m_1,m_2\in \mathbb N\cup\{0\}$ with $m_1+2m_2=i$.  Importantly, with $s=0$ we get $|x|^{1-2m_1-2m_2}(x\cdot z_1)^{m_1}|z_1|^{2m_2}$.  This can be shown through a simple induction argument.  Each term is a polynomial in $z_1$ of degree $m_1+2m_2=i$. In addition,
	\be \label{partialis}
	\big|\partial^i_s|x-sz_1|\big|\les |x-sz_1|^{1-i}|z_1|^i
	\ee
	Since $\partial_s f(r_1)=f'(r_1)\partial_sr_1$, each derivative acting on $f$ produces a factor of $\partial_s r_1$ by the Chain Rule, so that the total number of derivatives acting on $r_1$ (counting multiplicity) sums up to $j$.  So that
	$$
	\partial_s^j f(r_1)=\sum_{k=0}^{j} f^{(k)}(r_1)\prod_{i=1}^{I_k}(\partial_s^i r_1)^{m_i}
	$$	
	where $\sum m_i=k$ and $\sum im_i=j$ when summed over the index set $I_k$.  This may be proven by an inductive argument, noting that  if another derivative acts on a $(\partial_s^\ell r_1)^{m_\ell}$, only the indices $m_\ell$ and $m_{\ell+1}$ are affected, changing to $m_{\ell}-1$ and $m_{\ell+1}+1$ respectively.
	Using Lemma~\ref{prop:F} we see that, with $r_1=|x-sz_1|$, $\partial^j_s\mR_0(\lambda^{2m})(r_1)\rvert_{s=0}$ is a linear combination of terms of the form:
	\begin{multline}\label{eqn:mR_0_terms}
		r_1^{2m-n-j_1}\lambda^{j_2}e^{i\lambda r_1}\lambda^{j_3} F^{(j_3)}(\lambda r_1)\prod^{I}_{i=1}(\partial^i_sr_1)^{m_i}\Big\rvert_{s=0}\\
		= |x|^{2m-n-j_1}\lambda^{j_2}e^{i\lambda |x|}\lambda^{j_3}F^{(j_3)}(\lambda |x|)\prod^{I}_{i=1}\Big(\partial^i_s|x-sz_1|\Big\rvert_{s=0}\Big)^{m_i}
	\end{multline}
	with $\sum^I_{i=1}im_i=j$ and $\sum^I_{i=1}m_i=j_1+j_2+j_3$.  Technically, the index set $I$ should change in each occurrence, we leave it as is for simplicity.  The only factors of $z_1$ come from that rightmost product which is a polynomial of degree $\sum^I_{i=1}im_i=j$ in $z_1$, by \eqref{eqn:partials in s ugly} and the discussion following it.  Using this fact with our $k_1-1$ degrees of cancellation, since $\phi\in S_{k+1}$ yields that $\la v\phi, x^{\alpha}\ra=0$ for all $|\alpha|\leq k-1$ by \eqref{defn:Si}, in \eqref{eqn:mR_0expn} implies that  the contribution of the first $k_1-1$ terms to the $z_1$ integral are zero provided that we don't have the cutoff $\chi_{|z_1|\ll|x|}$. Therefore, we can replace the cutoff $\chi_{|z_1|\ll|x|}$ with $-\chi_{|z_1|\gtrsim|x|}$. We have 
	\begin{multline*}
		\int_{\R^n}\chi_{|z_1|\ll |x|} \mR_0^+(\lambda^{2m})(|x-z_1|)v\phi(z_1)\ \!dz_1=-\sum_{j=0}^{k_1-1}\int_{\R^n}\chi_{|z_1|\gtrsim|x|}  \frac{\partial_s^jf(0)}{(k_1-1)!} v\phi(z_1)\ \!dz_1
		\\+\int_{\R^n}\int^1_0\frac{(1-s)^{k_1-1}}{(k_1-1)!}\frac{\partial^{k_1}}{\partial s^{k_1}}\mR_0^+(\lambda^{2m})(|x-sz_1|)v\phi(z_1)\ \!ds dz_1.
	\end{multline*}	
	We rewrite   \eqref{eqn:mR_0_terms} without the evaluation at $s=0$:	 
	$$ r_1^{2m-n-j_1-j_2-j_3} (\lambda r_1)^{j_2+j_3} e^{i\lambda r_1} F^{(j_3)}(\lambda r_1)\prod^{I}_{i=1}(\partial^i_sr_1)^{m_i}$$
	Note that, by \eqref{partialis},
	$$
	r_1^{2m-n-j_1-j_2-j_3} \Big|\prod^{I}_{i=1}(\partial^i_sr_1)^{m_i}\Big|\les r_1^{2m-n-j_1-j_2-j_3}r_1^{\sum (1-i)m_i}|z_1|^{\sum im_i} = r_1^{2m-n-j}  |z_1|^j
	$$
	Therefore, 
	$$
	\partial^{j}_s\mR_0(\lambda^{2m})(r_1) = e^{i\lambda r_1} \frac{|z_1|^j}{r_1^{n-2m+j}}  G_j(\lambda r_1),
	$$
	where  (also using $j_2+j_3\leq j$  and  the bounds in \eqref{Fbounds})
	$$
	|G_j(\lambda r_1)|\les \la \lambda r_1\ra^{\frac{n+1}2-2m+j}. 
	$$
	Strictly speaking 	the factor $\frac{|z_1|^j}{r_1^{n-2m+j}} $ is a function of $z_1$ and $r_1$ satisfying the given bound. 
	
	Evaluating this at $s=0$ yields the terms in the first line of the formula in the assertion of the lemma. 
	The case $j=k_1$ yields the second line.   The $\lambda$ derivatives follow similarly by applying Lemma~\ref{prop:F}.
\end{proof}

The final lemma provides bounds for the resulting $\lambda$-integral after we employ the cancellation lemma above. To streamline the argument for the new singular terms, for $r_1,r_2>0$, we define 
\begin{align}\label{eqn:Lambda def}
	\Lambda_{ j,\eta}(r_1,r_2):=\bigg|\int_0^1 e^{i\lambda (r_1 \pm r_2 )}  \omega^{\eta }(\lambda)   \ \omega^{\frac{n+1}2 -2m+j}(\lambda,  r_1 )\ \omega^{\frac{1-n}2} (\lambda, r_2)  d\lambda \bigg|.
\end{align}
The following lemma on $\Lambda_{ j,\eta}$ will be crucial in the analysis of $K(x,y)$. 
\begin{lemma}\label{lem:Lambda} 
	For $j\in\{0,1,...,k_1\}$ and $\eta>-1$, we have the bounds 
	\begin{enumerate}[i)]
		\item $\Lambda_{ j,\eta}\les 1$ when  $r_1,r_2\les 1$, 
		\item $\Lambda_{ j,\eta}\les \int_0^1  \frac { \lambda^{\eta}   }{  \la\lambda r_2\ra^{ 2m-1-j} \la \lambda (r_1-r_2)\ra^2}     d\lambda$, when   $r_1 \approx r_2\gg 1 $, 
		\item $\Lambda_{ j,\eta}\les r_1^{j-2m-\f12 }+r_1^{-\eta-1}$ when $r_1 \gg \la r_2\ra$, and
		\item$\Lambda_{ j,\eta}\les  r_2^{-\eta-1}+ \la r_1\ra^{-\min(\eta+1, \frac{n}2+2m-j)}  \big(\tfrac{\la r_1\ra }{r_2}\big)^{n+\frac12},$ when $r_2\gg \la r_1\ra $.
	\end{enumerate}
\end{lemma}
We will apply this with $\eta= \alpha-1+k_2-\varepsilon+n-2m > -1 $, which is independent of $j$.  Also, although the bounds depend on $j$, the worst case is $j=k_1$.    

\begin{proof} We only consider the `-' sign.
	The first bound follows by noting that the magnitude of the integrand in \eqref{eqn:Lambda def} is bounded by $ \lambda^{\eta}$, which is integrable. 
	
	For the second bound, we integrate by parts twice when $\lambda |r_1-r_2|\gtrsim1$  and estimate directly when $\lambda |r_1-r_2|\ll1$ to obtain 
	\begin{multline*} 
		\Lambda_{ j,\eta}\les 
		\int_0^1 \lambda^{\eta} \chi(\lambda|r_1-r_2|)  \la\lambda r_2\ra^{1-2m+j}    d\lambda  
		+ \int_0^1 \frac {  \lambda^{\eta-2} \widetilde\chi(\lambda|r_1-r_2|)   \la\lambda r_2\ra^{1-2m+j}}{|r_1-r_2|^2}     d\lambda  \\
		\les  \int_0^1  \frac { \lambda^{\eta}   }{  \la\lambda r_2\ra^{ 2m-1-j} \la \lambda (r_1-r_2)\ra^2}     d\lambda. 
	\end{multline*}
	For the third bound, when $\lambda r_1\les 1$,   we bound the $\omega$ terms by 1,   and estimate the  integral  by 
	$$\int_0^{\frac1{r_1}} \lambda^\eta d\lambda\les r_1^{-\eta-1}
	$$
	When $\lambda r_1\gtrsim 1$, we integrate by parts for   $N=\lceil\frac{n}2\rceil+1$ times.  Noting that--due to the support conditions for $\omega$ terms and $ \widetilde\chi(\lambda r_1)$--there are no boundary terms, we obtain   
	\begin{multline*} 
		\frac1{|r_1- r_2|^{N}}\int_0^1 \Big|\partial_\lambda^{N} 
		\big[\widetilde\chi(\lambda r_1) \omega^{\eta}(\lambda)   \ \omega^{\frac{n+1}2 -2m+j}(\lambda,  r_1 )\ \omega^{\frac{1-n}2} (\lambda, r_2)\big] \Big| d\lambda\\ 	
		\les r_1^{-N}  
		\int_{\frac1{r_1}}^1  \lambda^{\eta-N}  \lambda^{\frac{n+1}2-2m+j}  r_1^{\frac{n+1}2-2m+j}   \la \lambda r_2\ra^{-\frac{n-1}2}  d\lambda\\
		\les r_1^{j-2m-\{\frac{n}{2}\}-\f12 }  
		\int_{\frac1{r_1}}^1  \lambda^{\eta+j-2m-\{\frac{n}{2}\}-\f12 }   
		d\lambda 	 \les   r_1^{j-2m-\f12 }+r_1^{-\eta-1}.
	\end{multline*}
	Here if $\eta+j-2m-\{\frac{n}{2}\}-\f12>-1$ we bound the $\lambda$ integral by one.  If $\eta+j-2m-\{\frac{n}{2}\}-\f12<-1$, we get $ r_1^{-\eta-1}$.
	
	It remains to consider the fourth inequality; when $r_2\gg \la r_1\ra$. When $\lambda r_2\les 1$, the claim follows by bounding the integrand by $\lambda^\eta$ as in the proof of third inequality above. When $\lambda r_2 \gtrsim 1$, we  integrate by parts $N=\lceil\frac{n}2\rceil+1$ times  to obtain 
	\begin{multline*} 
		\frac1{|r_1- r_2|^{N}}\int_0^1 \Big|\partial_\lambda^{N} 
		\big[\widetilde\chi(\lambda r_2) \omega^{\eta}(\lambda)   \ \omega^{\frac{n+1}2 -2m+j}(\lambda,  r_1 )\ \omega^{\frac{1-n}2} (\lambda, r_2)\big] \Big| d\lambda\\ 	
		\les \frac1{r_2^{N+\frac{n-1}2}} 
		\int_{\frac1{r_2}}^1  \lambda^{\eta-N-\frac{n-1}2} \la \lambda r_1\ra^{\frac{n+1}2-2m+j}    d\lambda \\
		\les \frac1{r_2^{N+\frac{n-1}2}} \int_{\frac1{r_2}}^{\frac2{\la r_1\ra}} \lambda^{\eta-N-\frac{n-1}2}\, d\lambda + \big(\tfrac{\la r_1\ra }{r_2}\big)^{N+\frac{n-1}2} \int_{\frac2{\la r_1\ra}}^{1} \lambda^{\eta+1-N-2m+j}\la   r_1\ra^{1-N -2m+j} \, d\lambda \\ 
		\les r_2^{-\eta-1}+\la r_1\ra^{-\eta-1} \big(\tfrac{\la r_1\ra }{r_2}\big)^{N+\frac{n-1}2} +\la r_1\ra^{-\eta-1} \big(\tfrac{\la r_1\ra }{r_2}\big)^{N+\frac{n-1}2}+\la r_1\ra^{1-N-2m+j} \big(\tfrac{\la r_1\ra }{r_2}\big)^{N+\frac{n-1}2}  
		\\ \les  r_2^{-\eta-1}+ \la r_1\ra^{-\min(\eta+1, \frac{n}2+2m-j)}  \big(\tfrac{\la r_1\ra }{r_2}\big)^{n+\frac12}.  
	\end{multline*} 
	The third line follows by considering cases based on the size of the exponents on $\lambda$, and the last inequality  follows by noting that $N+\frac{n-1}2\geq n+\frac12$.
\end{proof}

We are now ready to prove the main result of this section.
\begin{proof}[Proof of Proposition~\ref{prop:low tail eigenspace}]
	We start with the case $k_1>0$.

	Applying  Lemma~\ref{lem:right_side_cancel} in \eqref{Kdef_eigen_2} gives (with $\varphi_2(\cdot)=O(\la \cdot\ra^{-n-\varepsilon}$) 
	\begin{multline}\label{eqn:K improved}
		K(x,y)\\
		=\int_{\R^{2n}}\int^1_0 \omega^{\alpha-1}(\lambda)  \lambda^{k_2-\varepsilon}\mR_0^+(\lambda^{2m})(|x-z_1|)v\phi_1(z_1) \varphi(z_2)[\mR_0^+(\lambda^{2m})-\mR_0^-(\lambda^{2m})](r_2)\ \!d\lambda dz_1dz_2.
	\end{multline}
	The quantity $\varepsilon>0$ above will be chosen to be sufficiently small depending on the fixed index $p$ and the  values of other parameters in the proposition. In the case $k_2=0$, we take $\varepsilon=0$ and the equality holds with $\varphi=v\phi$, which is in $L^1$ provided that $\beta>\frac{n}2$; we will omit this issue in the rest of the proof.

	To study $K(x,y)$, \eqref{eqn:K improved}, we consider the regions  
	$$A:=\{(x,z_1): \la z_1\ra \gtrsim \la  x \ra  \}\,\,\text{ and }
	A^c=\{(x,z_1):  |x|\gg 1 ,  |z_1|  \ll  |x|\}.$$
	Note that on the set $A$ we have  $\la z_1 \ra \gtrsim \la x\ra,  \la r_1\ra$, and all restriction on $L^p$ boundedness comes from the integrability in $y$ variable (provided that $\beta$ is sufficiently large). On the set $A^c$ we will also utilize Lemma~\ref{lem:left_side_cancel} to gain additional powers of $|x|$ and/or $r_1$.

	To that end, we define $K_0(x,y)$ to be $K(x,y)$ restricted to the set $A$.  On the set $A^c$, we utilize   Lemma~\ref{lem:left_side_cancel} to replace the resolvent on the left with the sum of $k_1$ terms.  That is, we write
	$$
	K(x,y) =\sum_{j=0}^{k_1}K_{ j}(x,y).
	$$
	Now, using the representations in Lemma~\ref{prop:F} (with $r_1=|x-z_1|$ and $r_2:=|z_2-y|$) in  \eqref{eqn:K improved}, we write  $K_0 $ as the difference of  
	\begin{multline} \label{K0defini}
		K_{0,\pm}(x,y)= \\ \int_{\R^{2n}}  \frac{\chi_{\la z_1\ra \gtrsim \la  x \ra} v\phi_1(z_1)\varphi(z_2) }{r_1^{n-2m}   } \int_0^1 e^{i\lambda (r_1 \pm r_2 )}  \omega^{\alpha-1+k_2-\varepsilon+n-2m }(\lambda)   \ \omega^{\frac{n+1}2 -2m}(\lambda,  r_1 )\ \omega^{\frac{1-n}2} (\lambda, r_2)  d\lambda dz_1 dz_2.
	\end{multline}
	Using the bound in Lemma~\ref{lem:eigen_decay} and  $\la z_1\ra \gtrsim \la  x \ra$ we have $$|v(z_1)\phi(z_1)|\les \gen{z_1}^{-\frac\beta2+2m-n-\frac\beta2}=\gen{z_1}^{2m-n-\beta } \les \frac{1}{ \gen{z_1}^{n+} \la r_1\ra^{\beta-2m-}} .$$
	Using this, the bound $|\varphi(z_2)|\les \gen{z_2}^{-n-}$, and the definition of $\Lambda_{j,\eta}$ in \eqref{eqn:Lambda def} with 
	\be\label{eta}
	\eta= \alpha-1+k_2-\varepsilon+n-2m,
	\ee we have 
	$$
	|K_{0,\pm}(x,y)| 
	\les \int_{\R^{2n}} \frac{dz_1 dz_2}{ \gen{z_1}^{n+}\gen{z_2}^{n+}  }\frac{ \Lambda_{0,\eta}(r_1,r_2)}{ r_1^{n-2m}\la r_1\ra^{\beta-2m-}  } .
	$$
	Similarly for $1\leq j<k_1$, using Lemma~\ref{lem:left_side_cancel} on $z_1$ integral on $A^c$, $K_{j}$ is  the difference  of 
	\begin{multline} \label{Kjdefini}
		|K_{j,\pm}(x,y)|
		= \chi_{|x|\gg 1}\Big| \int_{\R^{2n}}\frac{\chi_{|z_1|\gtrsim |x|}|z_1|^j v\phi_1(z_1)\varphi(z_2)}{|x|^{n-2m+j}}\\ \int^1_0 e^{i\lambda (|x|\pm r_2)} \omega^{\alpha-1+k_2-\varepsilon+n-2m }(\lambda)\ \omega^{\frac{n+1}{2}-2m+j}(\lambda,|x|)    \ \omega^{\frac{1-n}2} (\lambda, r_2)  \ \!d\lambda dz_1dz_2\Big| \\
		\les \chi_{|x|\gg 1} \int_{\R^{2n}}\frac{\chi_{|z_1|\gtrsim |x|}  \gen{z_1}^{-n-}\gen{z_2}^{-n-} }{|x|^{n-2m+j} |x|^{\beta-2m-j-}} \Lambda_{j,\eta}(|x|,r_2) dz_1 dz_2 \\ \les \chi_{|x|\gg 1}\int_{\R^{2n}} \frac{dz_1 dz_2}{ \gen{z_1}^{n+}\gen{z_2}^{n+}  }\frac{ \Lambda_{j,\eta}(|x|,r_2)}{ |x|^{n-4m+\beta-}  },
	\end{multline}
	and $K_{ k_1}$  is  the difference  of (with $r_1=|x-sz_1|\approx |x|\gg 1$ and using $\beta>2m+k_1$)
	\begin{multline}\label{Kk1defini}
		|K_{k_1,\pm }(x,y)| =\chi_{|x|\gg 1} \Big|\int_{\R^{2n}}\frac{\chi_{|z_1|\ll |x|}|z_1|^{k_1} v\phi_1(z_1)\varphi(z_2)}{r_1^{n-2m+k_1}}\int^1_0\frac{(1-s)^{k_1-1}}{(k_1-1)!} \\
		\int^1_0 e^{i\lambda (r_1\pm r_2)} \omega^{\alpha-1+k_2-\varepsilon+n-2m }(\lambda)\ \omega^{\frac{n+1}{2}-2m+k_1}(\lambda,r_1)    \ \omega^{\frac{1-n}2} (\lambda, r_2)  \ \!d\lambda ds dz_1dz_2\Big|\\
		\les  \int_{\R^{2n}}\frac{\chi_{r_1\gg 1}  \gen{z_1}^{-n-}\gen{z_2}^{-n-}}{r_1^{n-2m+k_1} }\int^1_0  \Lambda_{k_1,\eta}(r_1,r_2)  ds dz_1 dz_2.
	\end{multline}
	Using the bounds $\beta>2m+k_1$ and $j<k_1$, and noting that the worst case in Lemma~\ref{lem:Lambda} is when $j=k_1$, the required bounds for $K_{j,\pm}$ follows from the bounds for $K_{k_1,\pm }$, replacing $r_1\gg1$ with $|x|\gg 1$ in the argument. 
	
	Also note that the contribution of the first bound in Lemma~\ref{lem:Lambda} is relevant only for $K_{0,\pm}$, and it is straight-forward to verify that the contribution to $K_{0,\pm}$ is admissible.\footnote{We say an operator $K$ with integral kernel $K(x,y)$ is admissible if
		$$
		\sup_{x\in \R^n} \int_{\R^n} |K(x,y)|\, dy+	\sup_{y\in \R^n} \int_{\R^n} |K(x,y)|\, dx<\infty.
		$$
		By the Schur test, it follows that an operator with admissible kernel is bounded on $L^p(\R^n)$ for all $1\leq p\leq \infty$.}  
	As for $j\in\{1,...,k-1\}$, the rest of the proof for $K_{0,\pm}$ follows from the bounds for $K_{k_1,\pm }$.

	Before we start, we note that
	by taking $\varepsilon<\min(k_2+\alpha, \alpha +k_1+k_2 +n-4m)$ (see  \eqref{k1k2alpha}) and using   \eqref{eta}, we have
	\be\label{etabound}
	\eta >0, \,\, \text{ and } \,\,  \eta+k_1-2m=   \alpha-1+k_1+k_2-\varepsilon+n-4m>-1.  
	\ee 
	
	We denote the contributions of second, third, and fourth bounds in Lemma~\ref{lem:Lambda} to $K_{k_1,\pm}$ by $K_{k_1,2}$, $K_{k_1,3}$, and  $K_{k_1,4}$, respectively.  We will ignore the integral in $s$ as the bounds will be uniform in $s$.  
	We have 
	$$
	\|K_{k_1,2}(x,y)\|_{L^1_y} \les  \int_{\R^{2n}} \frac{1 }{ \gen{z_1}^{n+}\gen{z_2}^{n+}   }  \int_0^1 \frac { \lambda^{\eta}    }{r_1^{n-2m+k_1}  \la\lambda r_1\ra^{ 2m-1-k_1} }   \Big\|\frac { \chi_{ r_2\approx r_1 \gg 1}   }{   \la \lambda (r_1-r_2)\ra^2}   \Big\|_{L^1_y}  d\lambda dz_1 dz_2.
	$$
	Passing to polar coordinates, we have the $L^1$ norm above is bounded by $r_1^{n-1}\lambda^{-1}$ which leads to the bound
	$$\les  \int_{\R^{2n}} \frac{1}{ \gen{z_1}^{n+}\gen{z_2}^{n+}   }  \int_0^1   \frac {\lambda^{\eta-1}   \chi_{r_1 \gg 1}   }{ r_1^{1-2m+k_1 }  \la\lambda r_1\ra^{ 2m-1-k_1}  } d\lambda dz_1 dz_2
	$$
	For $k_1\leq 2m-1$, we estimate this by 
	$$\les  \int_{\R^{2n}} \frac{1}{ \gen{z_1}^{n+}\gen{z_2}^{n+}   }  \int_0^1      \lambda^{\eta+k_1-2m}       d\lambda dz_1 dz_2\les 1.
	$$
	In the last inequality we used \eqref{etabound}.

	For $k_1\geq 2m$, we have the bound 
	$$
	\int_0^1 \frac{\lambda^{\eta-1} \chi_{r_1\gg1}}{r_1^{1-2m+k_1}   }  d\lambda  + \int_0^1  \lambda^{\eta+k_1-2m}   d\lambda \les 1.
	$$
	The last inequality follows from \eqref{etabound}.
	The $L^1_x$ norm is bounded similarly, and hence $K_{k_1,2}$ is admissible.

	We now consider $K_{k_1,3}$:
	$$
	|K_{k_1,3}(x,y)|\les  \int_{\R^{2n}} \frac{dz_1 dz_2}{ \gen{z_1}^{n+}\gen{z_2}^{n+}  }\frac{(r_1^{-\eta-1}+r_1^{k_1-2m-\f12 })\chi_{r_1\gg\la r_2\ra}}{r_1^{n-2m+k_1}  }.
	$$ 
	This bound yields an admissible kernel as  $\eta+k_1-2m+1>0$ by \eqref{etabound}. 
	
	It remains to consider $K_{k_1,4}$.  
	$$
	|K_{k_1,4}(x,y)|\les  \int_{\R^{2n}} \frac{ dz_1 dz_2}{ \gen{z_1}^{n+}\gen{z_2}^{n+}  }\frac{\chi_{r_2\gg r_1\gg 1}}{r_1^{n-2m+k_1}  } \big(r_2^{-\eta-1}+ r_1^{-\min(\eta+1, \frac{n}2+2m-k_1)}  \big(\tfrac{ r_1  }{r_2}\big)^{n+\frac12}\big) .
	$$ 
	Note that the contribution of the second term in the parenthesis is admissible as  
	$$n-2m+k_1+ \min(\eta+1, \frac{n}2+2m-k_1) > n, $$
	which follows from \eqref{etabound}. 
	
	The contribution of the first term is bounded in $L^p$ for $p<\frac{n}{n-\eta-1}=\frac{n}{2m-k_2-\alpha+\varepsilon}$ provided that 
	$$
	\eta+1+n-2m+k_1>n,
	$$
	which follows from \eqref{etabound}. For $k_2+\alpha\leq 2m$, this yields the claim by taking $\varepsilon$ sufficiently small. For $k_2+\alpha>2m$, we have an admissible kernel as $\eta+1>n$ in that case.
	
	It remains to consider  the case $k_1=0$ when $2m<n<4m$. 
	Note that, by taking $\varepsilon <\alpha+k_2-(4m-n)$, we have 
	$$
	\eta+1=\alpha+k_2 -\varepsilon+n-2m>2m.
	$$
	We have (see \eqref{eqn:K improved} and the discussion below it), with $\varphi_1(z_1)=|v(z_1)\phi_1(z_1)|\in L^1\cap L^2$ and  $\varphi_2(z_2)=|v(z_2)\phi_2(z_2)|\in L^1$ if $k_2=0$ and $\gen{z_2}^{-n-}$ otherwise,  
	$$
	|K(x,y)|\les \int_{\R^{2n}} \frac{  \varphi_1(z_1) \varphi_2(z_2)    }{r_1^{n-2m}   } \Lambda_{0,\eta}(r_1,r_2) dz_1 dz_2,
	$$
	provided that $\beta>n$.
	Using Lemma~\ref{lem:Lambda} and $\eta+1>2m$, we have 
	\begin{enumerate}[i)]
		\item $\Lambda_{ 0,\eta}\les 1$ when  $r_1,r_2\les 1$, 
		\item $\Lambda_{ 0,\eta}\les \int_0^1  \frac { \lambda^{\eta}   }{  \la\lambda r_2\ra^{ 2m-1 } \la \lambda (r_1-r_2)\ra^2}     d\lambda$, when   $r_1 \approx r_2\gg 1 $, 
		\item $\Lambda_{0,\eta}\les r_1^{-2m- }$ when $r_1 \gg \la r_2\ra$, and
		\item$\Lambda_{0,\eta}\les  r_2^{-\eta-1}+ \la r_1\ra^{-2m-}  \big(\tfrac{\la r_1\ra }{r_2}\big)^{n+\frac12},$ when $r_2\gg \la r_1\ra $.
	\end{enumerate}
	In case i), we have an admissible kernel as $\frac1{r_1^{n-2m}}$ is locally $L^2$ and hence $\varphi_1(z_1)r_1^{2m-n}$ is  $L^1$. 
	
	In case ii), as above, we have the following bound for the $L^1_y$ norm    
	$$
	\les  \int_{\R^{2n}}\varphi_1(z_1) \varphi_2(z_2)        \int_0^1  \frac { \lambda^{\eta-1}    \chi_{r_1\gg1}  }{   r_1^{1-2m}   \la\lambda r_1\ra^{ 2m-1}  }     d\lambda  dz_1 dz_2 \les \int_{\R^{2n}}\varphi_1(z_1) \varphi_2(z_2)        \int_0^1   \lambda^{\eta-2m}          d\lambda  dz_1 dz_2 \les 1, $$ 
	as $\eta-2m>-1$.  
	
	The case iii) is easily seen to be admissible, and the case iv) yields the same range of $p$ as above since $n-2m+\eta+1>n$.
\end{proof}

\section{Expansions for the resolvent and $M(\lambda)$}\label{sec:Minv}
In this section we develop an expansion of $M^\pm(\lambda)^{-1}$ for small $\lambda\in(0,\lambda_0)$ when there is a threshold eigenvalue.  Throughout this section we consider the `+' limiting operators and omit the superscript: $M(\lambda)=U+v\mR_0(\lambda^{2m})v$.
We utilize the Jensen-Nenciu inversion scheme, \cite{JN}, to understand $M^{-1}(\lambda)$ when there is a threshold obstruction. In contrast to the analysis for dispersive estimates, we seek to avoid lengthy expansions for the operators and their derivatives.

Recall that $T_0=U+v\mR_0(0)v=U+vG_0^0v$, where $G_0$ is the operator with integral kernel $G_0^0(x,y):=a_0|x-y|^{2m-n}$, and $S_1$ denotes the Riesz projection onto the kernel of $T_0$. Also recall the definitions  $S_{i+1}$ from \eqref{defn:Si}.  In the case $n>4m$, there are no resonances, and having an eigenvalue at zero means $T_0$ is not invertible, however, $T_1=S_1vG_1^0vS_1$ is invertible on $S_1L^2$.  Here the kernel of $G_1^0$ is a constant multiple of $|x-y|^{4m-n}$, see \cite{soffernew,EGWaveOp,egl} for further details. 

In the case $2m<n\leq 4m$, with $k=2m-\lceil\frac{n}{2}\rceil+1$, having only an eigenvalue at zero (no resonances) means $T_0$ is not invertible, however, $S_1=S_{k+1}$ and $T_{k+1}:= S_{ 1}vG^0_{1,L}vS_{ 1}$ is invertible on $S_1L^2$, where $G_{1,L}^0(x,y):= |x-y|^{4m-n}(b_{1}+b_{2}\log(|x-y|))$ ($b_2=0$ in odd dimensions). See  \cite{soffernew} for more details. 
We also define the operators with integral kernels
$$
G_j^1(x,y):= c_j|x-y|^{2j},\quad j=0,1,2,... 
$$
which will be important in the expansions for $M(\lambda)$. 	 
Finally, note that the operator $T_0+S_1$ is invertible in all cases and we denote its inverse  by $D_0$. We omit the proof of the standard fact that $D_0$ is absolutely bounded, and that $D_0S_1=S_1D_0=S_1$, see \cite{egl} for more details.

We introduce some notation to help streamline the upcoming statements and proofs.  
For an absolutely bounded operator $T(\lambda)$ on $L^2(\R^n)$ we write $T(\lambda)=O_N(\lambda^j)$ to mean that for $0\leq\ell\leq N$,
$$
\sup_{0<\lambda<\lambda_0}\lambda^{\ell-j}|\partial_\lambda^\ell T(\lambda)|
$$
is a bounded operator on $L^2$.

We note that, if $S,T$ satisfy $S(\lambda)=O_{N_1} (\lambda^{j})$ and $T(\lambda)=O_{N_2} (\lambda^{k})$, by the product rule the composition of the operators   is absolutely bounded and satisfies
$$
S(\lambda)T(\lambda)=O_{\min(N_1,N_2)} (\lambda^{j+k}).
$$
Similarly,
$$
\lambda^j T(\lambda)=O_{N_2}(\lambda^{j+k}).
$$
If $R$ is an absolutely bounded operator on $L^2$, then 
$$
S(\lambda)R,\, R S(\lambda)=O_{N_1}(\lambda^{j}).
$$
By Lemma~3.2 in \cite{egl}, we also have 
\be\label{ONinv}
[\Gamma+O_N(\lambda^{\epsilon})]^{-1}= \Gamma^{-1}+O_N(\lambda^{\epsilon}),
\ee
provided that $\Gamma$ is an invertible  $\lambda$-independent operator and $\epsilon>0$, and that $\lambda_0$ is sufficiently small. 

\begin{defin}
	Let $T(\lambda)$ be a $\lambda $ dependent absolutely bounded operator.  We say $T$ is an admissible error, and write $T=\mathcal{AE}$ if $T(\lambda)=O_0(\lambda^0)$ and $\partial_\lambda T(\lambda)=O_{\lceil \tfrac{n}2\rceil}(\lambda^0)$.
\end{defin} 
We note that $T=\mathcal{AE}$ is equivalent to the $L^2$-boundedness of the operator with kernel   
$$
\widetilde T (x,y):= \sup_{0<\lambda <\lambda_0}\Big[|T(\lambda)(x,y)|+ \sup_{1\leq \ell\leq  \lceil \tfrac{n}2\rceil +1  } \big|\lambda^{\ell-1} \partial_\lambda^\ell T(\lambda)(x,y)\big| \Big]. 
$$
Also note that $T=\mathcal{AE}$ under the more restrictive assumption $T=O_{ \lceil \tfrac{n}2\rceil +1}(\lambda)$.

Note that if $S,T=\mathcal{AE}$, then their composition $ST =\mathcal{AE}$. Finally,   if $T=\mathcal{AE}$ is invertible for each  $\lambda\in (0,\lambda_0)$ and if 
$ T^{-1}=O_0(\lambda^0) $, then $T^{-1}=\mathcal{AE}$.
To see this note that  
$$
\partial_\lambda T^{-1}(\lambda)=T^{-1}(\lambda)[\partial_\lambda T(\lambda)]T^{-1}(\lambda),
$$
so $\partial_\lambda^kT^{-1}(\lambda)$
is a linear combination of operators of the form:
$$
T^{-1}(\lambda)\prod_{k_j\geq 1, \sum k_j=k}\bigg([\partial_\lambda^{k_j}T(\lambda)]T^{-1}(\lambda)\bigg).
$$

In our usage, $T(\lambda)$ is an operator defined in terms of the free resolvent $\mR_0(\lambda^{2m})$, for which we prove various bounds.
This helps to minimize the required decay on $V$ and avoid lengthier expansions and streamline our analysis.

We recall the following from \cite{EGWaveOpExt}, see equation (14) and the bounds following it:
$$v \mathcal R_0(\lambda^{2m}) v =\mathcal{AE},$$
provided that $\beta>n_\star$. This implies that $M(\lambda)=U+v\mR_0(\lambda^{2m})v$ and $M(\lambda)+S_1$ are $\mathcal{AE}$ under the same condition.

We also recall Corollary~2.2 in \cite{JN} which tells us that 
\be\label{eqn:JN Minv id}
M^{-1}(\lambda)=[M (\lambda)+S_1]^{-1} - \lambda^{-2m}[M (\lambda)+S_1]^{-1} S_1 B^{-1}(\lambda)S_1[M (\lambda)+S_1]^{-1} ,
\ee
provided that 
\be\label{Blamdef}
B(\lambda)=\frac{1}{\lambda^{2m}}[S_1(M(\lambda)+S_1)^{-1}S_1 -S_1] 
\ee
is invertible on $S_1L^2$. 
By the assumption that there is a zero energy eigenvalue, $T_0$ is not invertible but $T_0+S_1$ is invertible on $L^2$.   Note that by a Neumann series expansion  and  the invertibility of $T_0+S_1$, the bound $(M(\lambda)+S_1)^{-1}=O_0(\lambda^0)$ follows from   the expansion
$$M(\lambda)+S_1=M(0)+S_1+[M(\lambda)-M(0)]=(T_0+S_1)+\int_0^\lambda \partial_s M(s) ds = T_0+S_1+ O_0(\lambda^1).$$
In the last inequality we used that $M(\lambda)=\mathcal{AE}$.

This implies that $[M(\lambda)+S_1]^{-1}=\mathcal{AE}$. The second summand in \eqref{eqn:JN Minv id} is more complicated and we first  consider the case $n>4m$. 

Let $M_1(\lambda)= \frac{(M(\lambda)+S_1)^{-1} -D_0}{\lambda^{2m}}$. Note that, using  $D_0S_1=S_1D_0=S_1$, we have $B(\lambda)=S_1M_1(\lambda)S_1$. We recall the following more detailed expansion for $(M(\lambda)+S_1)^{-1}$ from  \cite[Lemma 3.4]{egl}:\footnote{  In \cite{egl}, this lemma  was stated with   an error $\lambda^{2m+\epsilon}$, $0<\epsilon<1$.  However, the proof also yields the case $\epsilon =1$ as the logarithm in even dimensions doesn't appear in the first term of the error. This requires $\beta>n_\star+2$.} 
\be\label{eqn:M exp n big}
M(\lambda)=T_0+\lambda^{2m}vG_1^0v+O_{\lceil\frac{n}2\rceil+1}(\lambda^{2m+1}),
\ee
provided $\beta>n_\star+2$. 
So that,
\be\label{MS1inv}
(M(\lambda)+S_1)^{-1}=D_0-\lambda^{2m}D_0vG_1^0vD_0+O_{\lceil\frac{n}2\rceil+1}(\lambda^{2m+1}).
\ee
Therefore, we have $M_1=\mathcal{AE}$, and hence $B=\mathcal{AE}$. 

To invert $B(\lambda)$ on $S_1L^2$, recall that (see \cite{egl})  $T_1=S_1vG_1^0vS_1$ is invertible on $S_1L^2$ (for $n>4m$). Using \eqref{MS1inv}, we write 
\be\label{Bexp}
B(\lambda)= -T_1 + O_{\lceil\frac{n}2\rceil+1}(\lambda^1)
\ee
and conclude that $B^{-1}(\lambda)=O_0(\lambda^0)$ on $S_1L^2$, which implies that $B^{-1}(\lambda)=\mathcal{AE}$.
With this we may write the second summand in \eqref{eqn:JN Minv id} as 
$$
M_1(\lambda) S_1B^{-1}(\lambda) S_1 [M(\lambda)+S_1]^{-1}+ S_1B^{-1}(\lambda)S_1 M_1(\lambda)+ \frac{1}{\lambda^{2m}} S_1B^{-1}(\lambda)S_1.
$$
Since $M_1(\lambda), (M(\lambda)+S_1)^{-1}, S_1B^{-1}(\lambda)S_1$ are $\mathcal{AE}$, we conclude that the first two summands above are $\mathcal{AE}$,  and we can write  
\be \label{MinvAE}
M^{-1}(\lambda)=\mathcal{AE}  -\frac1{\lambda^{2m}}S_1B^{-1}(\lambda) S_1.
\ee 
Let $\{\phi_i:i\in\{1,...,d\}\}$ be an orthonormal basis for $S_1L^2$. By \eqref{MinvAE}, \eqref{Bexp}, and \eqref{ONinv}, we have 
$$
M^{-1}(\lambda)(x,y)= \mathcal{AE}+ \frac1{\lambda^{2m}} \sum_{i,j=1}^d
\phi_i(x)    \overline{\phi_j}(y) \omega^0(\lambda),
$$
where the operator in the  second term is of form $(k_0,k_0,0)$ when $S_1=S_{k_0+1}$. This yields the following proposition in the case $n>4m$: 
\begin{prop}\label{Minvsing} Let $n>2m$.
	Assume that $S_1=S_{k_0+1}$ for some $\max(k,1)\leq k_0\leq 2m+1$, then
	$$
	M^{-1}(\lambda)=M_e(\lambda)+\lambda^{-2m}\widetilde M(\lambda),
	$$
	where $M_e(\lambda)=\mathcal{AE}$ and
	$\widetilde M(\lambda)$ is a finite linear combination of operators of the form $(k_1,k_2,\alpha)$, where 
	$$
	(k_1,k_2,\alpha)\in\{ (k_0,k_0,0),(0,k_0,k_0),(k_0,0,k_0)\},
	$$ 
	provided that $\beta>\max(8m-n,n_\star)+2$. Moreover, when $n\geq 4m$, $\widetilde M(\lambda)=S_1B^{-1}(\lambda)S_1$, that is $(k_1,k_2,\alpha)=(k_0,k_0,0)$.
	
\end{prop}

To invert the operator $B(\lambda)$ (see   \eqref{Blamdef} for the definition) and prove Proposition~\ref{Minvsing}  in the case $2m<n\leq 4m $, we need more detailed expansions of $M(\lambda)$:
\begin{lemma}\label{lem:R0 for Mexp_small}
	When $2m<n\leq 4m$ and for $0<\lambda<\lambda_0$, we have  (with $k=2m-\lceil\frac{n}{2}\rceil+1$)
	\begin{align}\label{Mexpmidd}
		M(\lambda)= T_0+\lambda^{2m}vG_{1,L}^0v+\lambda^{2m}\log(\lambda)vG_1^0v
		+\sum^{k-1}_{j=0}\lambda^{n-2m+2j}vG_j^1v+O_{\lceil\frac{n}{2}\rceil+1}(\lambda^{2m+\epsilon}),
	\end{align}
	for any $0\leq\epsilon\leq 1$
	provided $\beta>\max(8m-n,n_\star)+2\epsilon$.
\end{lemma}
We note that the finite sum   contributes powers of $\lambda$ between zero and $2m$, the smallest of which is $\lambda^{n-2m}$, which complicates matters.

\begin{proof}
	We need to consider cases as the resolvents behave differently in even and odd dimensions and based on the size of $\lambda|x-y|$.
	When $\lambda|x-y|\ll 1$, we have (see  Lemma~2.3, equation (2.2), and Remark 2.2 in \cite{soffernew}) that
	\begin{align*}
		\mR_0(\lambda^{2m})(x,y)=&a_0|x-y|^{2m-n}
		+\sum^\infty_{j=0}b_j\lambda^{n-2m+2j}|x-y|^{2j}\\
		&+\sum^\infty_{j=1}\lambda^{2mj}|x-y|^{2m(j+1)-n}\Big(c_{j,1}+c_{j,2}\log(|x-y|)+c_{j,3}\log(\lambda)\Big).
	\end{align*}
	When $n$ even, the first logarithmic term in the expansion occurs with a coefficient of $\lambda^{2m}$  causing the additional leading term in the statement of the Lemma.  In particular, if $n$ is odd $c_{j,2}=c_{j,3}=0$ for all $j$. Let 
	$$
	E(\lambda)= \sum^\infty_{j=k}b_j\lambda^{n-2m+2j}|x-y|^{2j}+\sum^\infty_{j=2}\lambda^{2mj}|x-y|^{2m(j+1)-n}\Big(c_{j,1}+c_{j,2}\log(|x-y|)+c_{j,3}\log(\lambda)\Big).
	$$
	Since $\lambda|x-y|\ll 1$, we may divide by positive powers of $\lambda |x-y|$ ($n+2-2m-\epsilon$ and $2m-\epsilon$ respectively) to  dominate $E(\lambda)$ by  $\lambda^{2m+\epsilon}|x-y|^{4m-n+\epsilon}$ since $n-2m+2k=2m+2-2\{\frac{n}{2}\}>2m$.  In this regime, differentiation is comparable to division by $\lambda$.
	
	When $\lambda|x-y|\gtrsim 1$, the bound on the error term when $\lambda|x-y|\gtrsim 1$ follows from the bounds in Lemma~2.3 of \cite{EGWaveOp} and the definition of the kernels of $G_j$.  We use $\log(\lambda|x-y|)\lesssim(\lambda|x-y|)^{0+}$ and note
	\begin{align*}
		\lambda^{\ell}|\partial_\lambda^\ell E(\lambda)|=&\lambda^{\ell}\partial_\lambda^\ell\Big|\mR_0(\lambda^{2m})-G_0^0(x,y)-\lambda^{2m}\log(\lambda)G_1^0(x,y)-\lambda^{2m}G_{1,L}^0-\sum^{k-1}_{j=0}\lambda^{n-2m+2j}G_j^1\Big|\\
		\les&\lambda^{\frac{n+1}{2}-2m+\ell} |x-y|^{\frac{1-n}{2}+\ell}+\lambda^{2m+}|x-y|^{4m-n+}+\lambda^{2m-2\{\frac{n}{2}\}}|x-y|^{4m-2\lceil\frac{n}{2}\rceil}.
	\end{align*}
	We note that $\frac{1-n}{2}+\ell\leq \frac{1-n}{2}+\lceil \frac{n}2\rceil+1=\{\frac{n}2\}+\frac32$, we may write the first term as $\lambda^{\frac{n-1}{2}-2m+\ell} |x-y|^{\frac{1-n}{2}+\ell} $, and multiply by $(\lambda|x-y|)^{\alpha+}\gtrsim 1$ for $\alpha=4m-(\frac{n-1}{2}+\ell)+\epsilon>0$ to see 
	$$
	\lambda^{\ell}|\partial_\lambda^\ell E(\lambda)\widetilde\chi(\lambda|x-y|)|\les \lambda^{2m+\epsilon}[|x-y|^{\{\frac{n}2\}+\frac32+\epsilon}+|x-y|^{4m-n+\epsilon}].
	$$
	
	The same result follows for $n$ odd with the absence of the logarithmic terms, noting that $2(k-1)=4m+2-2\lceil\frac{n}{2}\rceil-2=4m-n-1$ in this case.
	
	From the definition of $E(\lambda)$, we have
	\begin{multline*}
		M(\lambda)= U+v\mR_0(\lambda^{2m})v \\ =U+vG_0^0v+\lambda^{2m}vG_{1,L}^0v+\lambda^{2m}\log(\lambda)vG_1^0v+\sum^{k-1}_{j=0}\lambda^{n-2m+2j}vG_j^1v+vE(\lambda)v,
	\end{multline*}
	where for $0\leq \ell \leq \lceil\frac{n}2\rceil+1$
	$$
	\sup_{0<\lambda<\lambda_0}\lambda^{\ell-2m-\epsilon}|\partial_\lambda^\ell vE(\lambda)v|\les v(x)(|x-y|^{\{\frac{n}2\}+\epsilon}+|x-y|^{4m-n+\epsilon})v(y).
	$$
	The first term is bounded provided $\beta>n_{\star}+2\epsilon$. The second term is Hilbert-Schmidt provided $\beta>8m-n+2\epsilon$ since $n\leq 4m$.  
\end{proof}

In particular, we note that Lemma~\ref{lem:R0 for Mexp_small} yields that $M(\lambda)+S_1=(T_0+S_1)+O_{\lceil\frac{n}{2}\rceil+1}(\lambda^{\min(2m-,n-2m)})$.   
To  invert the operator $B(\lambda)$ (see \eqref{Blamdef}),
$$
B(\lambda)=\frac{1}{\lambda^{2m}}(S_{ 1}(M(\lambda)+S_{ 1})^{-1}S_{ 1}-S_{ 1}),
$$
on $S_1L^2$ in lower dimensions, one must capture the natural orthogonality of $S_1=S_{k+1}$.  In fact, we will work under the assumption that $S_1=S_{k_0+1}$ for some $k_0\geq k$, where $k\leq k_0\leq 2m+1$ is as in Theorem~\ref{thm:main_low}. We have the following series of lemmas.
\begin{lemma}\label{lem:cancel} Let $  k_0\geq 1$.  Then for $j\leq\frac{k_0-1}{2}$
	\[S_{k_0+1}vG_j^1v=vG_j^1vS_{k_0+1}=0.\]
	Moreover, for $j\leq k_0-1$ 
	\[S_{k_0+1}vG_j^1v S_{k_0+1}=0.\]
\end{lemma}
For a proof of this standard orthogonality lemma see, e.g., Lemma~8.9 of \cite{soffernew}. 
\begin{lemma}\label{lem:one_side_cancel} Assume that $S_1=S_{k_0+1}$ for some $k_0\geq k$, where $k\leq k_0\leq 2m+1$, and $\beta>\max(8m-n,n_*)+2$.  
	Let  (the $\omega^{2m-}(\lambda)$ term only appears when $n$ is even)
	$$
	M_{ r}(\lambda):=  \sum_{\frac{k_0-1}2< j\leq k-1} \lambda^{n-2m+2j}vG_j^1vS_{ 1}  + \omega^{2m-}(\lambda) vG_1^0vS_{ 1}.
	$$
	We have
	\be\label{MS_1detail}	
	M(\lambda)S_1 = M_{ r}(\lambda) + \lambda^{2m}vG_{1,L}^0vS_{ 1} +O_{\lceil\frac{n}{2}\rceil+1}(\lambda^{2m+1})  
	\ee
	and 	$M(\lambda) S_{ 1}=O_{\lceil\frac{n}{2}\rceil+1}(\lambda^{ \alpha_0})$, where $\alpha_0=\min(n-2m+2\lfloor\frac{k_0-1}2\rfloor+2 ,2m-)$.
	The  analogous claim for $S_1M(\lambda)$ holds with 
	$$M_{ \ell}(\lambda) := \sum_{\frac{k_0-1}2< j\leq k-1} \lambda^{n-2m+2j}S_{ 1} vG_j^1v + \omega^{2m-}(\lambda)S_{ 1} vG_1^0v.$$
	In particular, $M_{ r}=M_{ \ell}=0$ if $k_0\geq 2 k-1=4m-2\lceil\tfrac{n}2\rceil+1$.
	Finally,
	$$
	S_{ 1} M(\lambda) S_{ 1}= \lambda^{2m}T_{k+1} + O_{\lceil\frac{n}{2}\rceil+1}(\lambda^{2m+1})=O_{\lceil\frac{n}{2}\rceil+1}(\lambda^{2m}).
	$$
\end{lemma}
\begin{proof} Using \eqref{Mexpmidd} with $\epsilon=1$, the cancellation property in Lemma~\ref{lem:cancel}, and $T_0S_1=0$, we obtain the first claim.
	Note that  when $n$ is odd, $\omega^{2m-}(\lambda)$ term  doesn't exist.
	The claim 	$M(\lambda) S_{ 1}=O_{\lceil\frac{n}{2}\rceil+1}(\lambda^{ \alpha_0})$ holds by noting that the smallest value of $j$ in the definition of $M_r$ is $\lfloor\tfrac{k_0-1}2\rfloor+1$. 
	
	The claim $M_{ r}=M_{ \ell}=0$ if $k_0\geq 2 k-1=4m-2\lceil\tfrac{n}2\rceil+1$ holds by noting that the finite sum is empty and the kernel of $vG_1^0v$ is 
	$$ vG_1^0v(x,y)=v(x)|x-y|^{4m-n}v(y)=v(x)|x-y|^{2(k-1)}v(y),$$
	as when $n$ is even $k=2m-\frac{n}{2}+1$. This contribution of this term is zero due to the projection on either side.  When $n$ is odd, this operator doesn't exist. 
	The final claim holds by using the second claim of Lemma~\ref{lem:cancel}.
\end{proof}
Here we note that $\alpha_0\geq m+1$, which can be seen by considering the cases $n (\text{mod }4) =0,1,2,3$ separately.\footnote{In the case $m=1$ and $n=4$, we only have $\alpha_0\geq 2- =m+1-$ unless $k_0=2m+1=3$. We omit this since the cases $k_0=k=1, k_0=2$ were already considered in \cite{Yaj4d}. }

We recall that $[M(\lambda)+S_1]^{-1}$ is $\mathcal{AE}$.
By the resolvent identity and Lemma~\ref{lem:one_side_cancel}, one has (recalling   $D_0=(T_0+S_1)^{-1}$, $T_0S_1=S_1T_0=0$,  $S_1D_0=D_0S_1=S_1$)
\begin{multline}\label{eqn:S1 Minv low d1}
	([M(\lambda)+S_1]^{-1}-D_0)S_1=-(M(\lambda)+S_1)^{-1}[ M(\lambda)+S_1 -T_0-S_1)]D_0S_1\\
	=-(M(\lambda)+S_1)^{-1}  M(\lambda) S_1 
	=(M(\lambda)+S_1)^{-1}O_{\lceil\frac{n}{2}\rceil+1}(\lambda^{\alpha_0}) = O_{\lceil\frac{n}{2}\rceil+1}(\lambda^{\alpha_0}) .
\end{multline}
Similarly, $
S_1([M(\lambda)+S_1]^{-1}-D_0) 
=O_{\lceil\frac{n}{2}\rceil+1}(\lambda^{\alpha_0})$. 
Using these, and the  resolvent identity as above, we have 
\begin{multline}\label{Blamsmalln}
	B(\lambda)=\frac1{\lambda^{2m}} S_1[(M(\lambda)+S_1)^{-1}-D_0]S_1
	=-\frac1{\lambda^{2m}}S_1 (M(\lambda)+S_1)^{-1}M(\lambda) S_1\\
	=-\frac1{\lambda^{2m}}S_1 M(\lambda )S_1 - \frac1{\lambda^{2m}}S_1[(M(\lambda)+S_1)^{-1}-D_0] M(\lambda) S_1\\
	=-\frac1{\lambda^{2m}} S_1 M(\lambda )S_1 - \frac1{\lambda^{2m}}S_1M(\lambda)  (M(\lambda)+S_1)^{-1} M(\lambda) S_1\\
	=-T_{k+1} +\frac1{\lambda^{2m}}S_1O_{\lceil\frac{n}{2}\rceil+1}(\lambda^{2m+1}) S_1+\frac1{\lambda^{2m}}S_1O_{\lceil\frac{n}{2}\rceil+1}(\lambda^{\alpha_0}) (M(\lambda)+S_1)^{-1}O_{\lceil\frac{n}{2}\rceil+1}(\lambda^{\alpha_0})S_1\\
	=-T_{k+1} +S_1O_{\lceil\frac{n}{2}\rceil+1}(\lambda^{1})S_1=\mathcal{AE}.
\end{multline} 
Therefore, by \eqref{ONinv}, 
$$S_1 B^{-1} S_1= - S_1T_{k+1}^{-1}S_1+ S_1 O_{\lceil\frac{n}{2}\rceil+1}(\lambda^{1})S_1 = S_1 O_{\lceil\frac{n}{2}\rceil+1}(\lambda^{0})S_1.$$
This also implies that $S_1 B^{-1} S_1=\mathcal{AE}$.
The contribution of $S_1 B^{-1} S_1$ to $  M^{-1}(\lambda)$ is 
\be\label{MBsmalln}
- \frac1{\lambda^{2m}} (M(\lambda)+S_1)^{-1} S_1 B^{-1}(\lambda) S_1  (M(\lambda)+S_1)^{-1}.
\ee
Note that, by Lemma~\ref{lem:cancel}, we have  
\be\label{Mk0AE}
\frac{M(\lambda)S_1}{\lambda^{2m}} = \frac{M_{r}(\lambda)}{\lambda^{2m}} +\mathcal{AE},
\ee
and if $k_0\geq 2k-1$, then $\frac{M(\lambda)S_1}{\lambda^{2m}}=\mathcal{AE}$.

Using \eqref{eqn:S1 Minv low d1}, and then \eqref{Mk0AE} and that $(M(\lambda)+S_1)^{-1}$ is $\mathcal{AE}$, we write
$$
\frac1{\lambda^{2m}} (M(\lambda)+S_1)^{-1} S_1
=\frac1{\lambda^{2m}} S_1 + \mathcal{AE}-(M(\lambda)+S_1)^{-1}  \frac{M_{r}(\lambda)}{\lambda^{2m}} .  
$$ 
We therefore have (recalling that $S_1B^{-1}S_1$ and $(M(\lambda)+S_1)^{-1}$  are $\mathcal{AE}$) 
\begin{multline*}
	\eqref{MBsmalln}= \mathcal{AE} -  S_1  B^{-1}(\lambda) S_1\frac{ S_1  (M(\lambda)+S_1)^{-1} }{\lambda^{2m}}+\\
	+(M(\lambda)+S_1)^{-1} M_{r}(\lambda)  S_1  B^{-1}(\lambda) S_1 \frac{S_1  (M(\lambda)+S_1)^{-1}}{\lambda^{2m}}.  
\end{multline*}
Repeating the argument above this time for $\frac{1}{\lambda^{2m}}S_1(M(\lambda)+S)^{-1}$ and noting that $M_{r},M_\ell=\mathcal{AE}$ and $\frac{M_{r}S_1B^{-1}(\lambda)S_1M_\ell}{\lambda^{2m}}=\mathcal{AE}$ since $\alpha_0\geq m+1$, we obtain
\begin{multline*}
	\eqref{MBsmalln}= \mathcal{AE} -   \frac{1}{\lambda^{2m}}S_1  B^{-1}(\lambda)  S_1+\frac{1}{\lambda^{2m}}S_1  B^{-1}(\lambda)  S_1 M_{ \ell} (M(\lambda)+S_1)^{-1} \\
	+\frac{1}{\lambda^{2m}}(M(\lambda)+S_1)^{-1} M_{ r}(\lambda) S_1  B^{-1}(\lambda)  S_1  \\
	-\frac{1}{\lambda^{2m}} (M(\lambda)+S_1)^{-1} M_{ r}(\lambda)  S_1  B^{-1}(\lambda) S_1 M_{ \ell} (\lambda) (M(\lambda)+S_1)^{-1}\\
	= \mathcal{AE} -   \frac{1}{\lambda^{2m}}S_1  B^{-1}(\lambda)  S_1+\frac{1}{\lambda^{2m}}S_1  B^{-1}(\lambda)  S_1 M_{ \ell} (\lambda) (M(\lambda)+S_1)^{-1} \\
	+\frac{1}{\lambda^{2m}}(M(\lambda)+S_1)^{-1} M_{ r}(\lambda) S_1  B^{-1}(\lambda)  S_1 .  
\end{multline*}
In particular, if $k_0\geq 2k-1$, $M_r=M_\ell=0$, so we have 
$$
M^{-1}(\lambda)=\mathcal{AE}-\frac1{\lambda^{2m}}S_1B^{-1}S_1,
$$
proving the claim in this case. Note that this completes the proof of Proposition~\ref{Minvsing} when $n=4m$ as $k=1$ and $k_0\geq k=1$ implies $k_0\geq 2k-1=1$. 

For $k\leq k_0\leq 2k-2$, recall that $M_{r}=O_{\lceil\frac{n}{2}\rceil+1}(\lambda^{\alpha_0})$. By a Neumann series expansion (with $N\geq m$), we have
$$
\frac{1}{\lambda^{2m}}(M(\lambda)+S_1)^{-1} M_{r}(\lambda) = 
\mathcal{AE}+\frac{1}{\lambda^{2m}}\Big[\sum_{i=0}^N (-1)^i D_0[(M(\lambda)-T_0)D_0]^i\Big]M_{r}(\lambda),
$$
as $M(\lambda)-T_0=O_{\lceil\frac{n}{2}\rceil+1}(\lambda^{ 1})$, and hence the remainder from the Neumann series expansion is $\mathcal {AE}$: 
$$[(M(\lambda)-T_0)D_0]^N (M(\lambda)+S_1)^{-1} M_{r}(\lambda) =O_{\lceil\frac{n}{2}\rceil+1}(\lambda^{ N+\alpha_0})=O_{\lceil\frac{n}{2}\rceil+1}(\lambda^{ 2m+1})=\mathcal{AE}.$$
Now, any term in $M(\lambda)-T_0$ that is $O_{\lceil\frac{n}2\rceil+1}(\lambda^{2m+1-\alpha_0})$ leads to an $\mathcal {AE}$ operator,  see \eqref{MS_1detail},  we  have
\begin{multline*}
	\frac{1}{\lambda^{2m}}(M(\lambda)+S_1)^{-1} M_{r}(\lambda) = 
	\mathcal{AE}+\frac{1}{\lambda^{2m}} D_0 M_{r}(\lambda) \\+ \frac{1}{\lambda^{2m}} \Big[\sum_{i=1}^N (-1)^i D_0\big[\sum_{j=0}^{k-1}\lambda^{n-2m+2j} vG_j^1v D_0\big]^i\Big]\sum_{\frac{k_0-1}2< j\leq k-1} \lambda^{n-2m+2j}vG_j^1vS_{1} .
\end{multline*}
Using this, and the analogous expansion with $M_\ell$, we  write  (for $k\leq k_0\leq 2k-2$)
$$
M^{-1}(\lambda)=\mathcal{AE}+\lambda^{-2m}\widetilde M(\lambda),
$$
where 
$\widetilde M$ is a finite linear combination of operators of the form $(k_1,k_2,\alpha)$, where 
$$
(k_1,k_2,\alpha)\in\{ (k_0,k_0,0),(0,k_0,\alpha_0),(k_0,0,\alpha_0)\},
$$ 
where $\alpha_0=n-2m+2\lfloor\frac{k_0-1}2\rfloor+2$ when $n$ is odd and $\alpha_0=\min(n-2m+2\lfloor\frac{k_0-1}2\rfloor+2 ,2m-)$ when $n$ is even. Note that 
$ \alpha_0\geq k_0$ since $k_0\leq 2k-2$, which finishes the proof of Proposition~\ref{Minvsing}. 

\section{Proof of the representation \eqref{Gammakrep} for $\Gamma_\kappa(\lambda)$}
In this section, using Proposition~\ref{Minvsing}, we will obtain the representation \eqref{Gammakrep} for $\Gamma_\kappa(\lambda)$.  Note that for $2m<n<4m$, Proposition~\ref{Minvsing} already implies \eqref{Gammakrep} for $\kappa=0$, though we will establish it for arbitrary $\kappa$.  When $n\geq 4m$, we also need pointwise bounds for the error term in \eqref{Gammakrep}.
If $n\geq4m$ the resolvents are not locally $L^2$ and one must have $\kappa$ sufficiently large to ensure the iterated resolvents are locally $L^2$ and achieve the needed pointwise bounds. 
We write the iterated resolvents 
\begin{align}\label{eq:Alambda}
	A(\lambda, z_1,z_2) =  \big[ \big(\mR_0^+(\lambda^{2m})  V\big)^{\kappa-1}\mR_0^+(\lambda^{2m})\big](z_1,z_2).
\end{align}  
For any $n\geq4m$.  From \cite{EGWaveOp,EGWaveOpExt}, for sufficiently large $\kappa$, we have
\begin{align*}
	\sup_{0\leq \ell \leq \lceil\frac{n}2\rceil+1}
	\sup_{0<\lambda <1}| \lambda^{\max\{0,\ell-1\}} \partial_\lambda^\ell 	A(\lambda, z_1,z_2)|&\les \la z_1 \ra^{\{\frac{n}2\}+\f32}  \la z_2\ra^{\{\frac{n}2\}+\f32}.
\end{align*}
This implies that the contribution of the $\mathcal {AE}$ portion of $M^{-1}(\lambda)$ from Proposition~\ref{Minvsing} to $\Gamma_{\kappa}(\lambda)$,
$$
UvA(\lambda)v M_e(\lambda)  vA(\lambda)vU,
$$
satisfies the hypotheses of Proposition~\ref{prop:low tail low d}.

The contribution of the singular part of $M^{-1}(\lambda)$, $\lambda^{-2m}S_1B^{-1}(\lambda)S_1$, to $\Gamma_{\kappa}(\lambda)$ is
$$
\lambda^{-2m} UvA(\lambda)v S_1B^{-1}(\lambda)S_1vA(\lambda)vU.
$$
We recall that for all $n>2m$ (see \cite{EGWaveOpExt} equation (14) and the bounds following it)
\begin{multline} \label{suplR0}
	\sup_{0\leq \ell \leq \lceil\frac{n}2\rceil+1}
	\sup_{0<\lambda <1}| \lambda^{\max\{0,\ell-1\}} \partial_\lambda^\ell 	\mR_0(\lambda^{2m})(x,y)| \\
	\les |x-y|^{2m-n}+ |x-y|^{2m-n+1}+|x-y|^{\{\frac{n}2\}+\frac32}+|x-y|^{-\frac{n-1}2}
	\\
	\les  |x-y|^{2m-n}+  |x-y|^{\{\frac{n}2\}+\frac32}. 
\end{multline}
Obviously, $G_0^0=\mR_0(0)$  satisfies the same bound. 
We also have for $n>4m$ (see \cite[Lemma 3.1]{egl}) 
\be\label{l2mR0}
\mR_0(\lambda^{2m})(x,y)=G_0^0(x,y)+\lambda^{2m}G_0^1(x,y)+\lambda^{2m}E(\lambda)(x,y),
\ee
where, for $0\leq \ell \leq \lceil\frac{n}2\rceil+1$, 
\be\label{l2mR0E}
\lambda^{\ell}|\partial_\lambda^{\ell}E(\lambda)(x,y)|\les \lambda [ |x-y|^{\{\frac{n}2\}+\frac32}+|x-y|^{4m-n+1}].
\ee
This implies 
\begin{multline*}
	\sup_{0\leq \ell \leq \lceil\frac{n}2\rceil+1}
	\sup_{0<\lambda <1}\big| \lambda^{\max\{0,\ell-1\}} \partial_\lambda^\ell 	\big(\lambda^{-2m} (\mR_0(\lambda^{2m})(x,y)-G_0^0(x,y) )\big)\big| \\ \les  |x-y|^{\{\frac{n}2\}+\frac32}+|x-y|^{4m-n+1}\les |x-y|^{2m-n}+  |x-y|^{\{\frac{n}2\}+\frac32}.
\end{multline*}
These imply that (see the proof of Proposition 5.3 in \cite{EGWaveOp})
\be\label{eqn:A diff L2}
\Big\|\sup_{0\leq \ell \leq \lceil\frac{n}2\rceil+1}
\sup_{0<\lambda <1}\big| \lambda^{\max\{0,\ell-1\}} \partial_\lambda^\ell 	 \big[v\frac{A(\lambda)-A(0)}{\lambda^{2m}}v\big](z_1,z) \big|  \Big\|_{L^2_z} \les \la z_1\ra^{-\frac{n}2-},
\ee
provided that $\beta>n_\star+2$.  Also note that $v[A(\lambda)-A(0)]v$ and $v A(\lambda)v$ satisfy the same bound. Therefore,
$$
\lambda^{-2m} UvA(\lambda)v S_1B^{-1}(\lambda)S_1vA(\lambda)vU = \lambda^{-2m} UvA(0)v S_1B^{-1}(\lambda)S_1vA(0)vU+\Gamma(\lambda),
$$
where $\Gamma(\lambda)$ satisfies the hypothesis of Proposition~\ref{prop:low tail low d}.

Finally, using $S_1T_0=0$ we have $S_1=-S_1vG_0^0vU=-UvG_0^0vS_1$, and hence
$$S_1vA(0)vU=S_1 v (G_0^0V)^{\kappa-1}G_0^0vU=S_1(vG_0^0vU)^{\kappa}=(-1)^\kappa S_1.$$
Using a similar identity on the left hand side, we have 
$$
\lambda^{-2m} UvA(\lambda)v S_1B^{-1}(\lambda)S_1vA(\lambda)vU = \lambda^{-2m}  S_1B^{-1}(\lambda)S_1 +\Gamma(\lambda).
$$
This yields the representation \eqref{Gammakrep} when $n>4m$.

Finally, we must take some care in the case of $n=4m$. The bound  \eqref{suplR0} is valid, however, for the error bound \eqref{l2mR0E} to hold we must account for the factor of size $\lambda^{2m}\log(\lambda)$ in the expansion of $\mR_0(\lambda^{2m})$ when $n=4m$ (see the proof of Lemma~\ref{lem:R0 for Mexp_small}).  To account for this,  we need to redefine $E(\lambda)$ as 
$$E(\lambda)=\frac{\mR_0(\lambda^{2m})-G_0^0-\lambda^{2m} G^0_{1,L}-\lambda^{2m}\log(\lambda)G_1^0}{\lambda^{2m}}.
$$
Here we note that when $n=4m$ the integral kernel of $G_1^0$ is a multiple of the constant function, which is annihilated by $S_1v$, i.e. $S_1vG_1^0=G_1^0vS_1=0$, see Lemma~\ref{lem:cancel}.
To utilize this, we write
\begin{multline}\label{Al-A0}
	A(\lambda)-A(0) =(\mR_0(\lambda^{2m})V)^{\kappa-1}\mR_0(\lambda^{2m})-(G_0^0V)^{\kappa-1}G_0^0\\
	=\sum_{j=1}^{\kappa}(G_0^0V)^{j-1} [\mR_0(\lambda^{2m})-G_0^0](V\mR_0(\lambda^{2m}))^{\kappa-j}.
\end{multline}
Therefore,
\begin{multline*}
	S_1v\frac{A(\lambda)-A(0)}{\lambda^{2m}}   =\sum_{j=1}^{\kappa}S_1v(G_0^0V)^{j-1} \frac{\mR_0(\lambda^{2m})-G_0^0}{\lambda^{2m}}  (V\mR_0(\lambda^{2m}))^{\kappa-j} \\
	=\sum_{j=1}^{\kappa}(-1)^{j-1}S_1v \frac{\mR_0(\lambda^{2m})-G_0^0}{\lambda^{2m}}  (V\mR_0(\lambda^{2m}))^{\kappa-j}
	= S_1 \sum_{j=1}^{\kappa}(UvG_0^0v)^{j-1} S_1v E(\lambda) (V\mR_0(\lambda^{2m}))^{\kappa-j}.
\end{multline*}
Therefore, the previous bound on $v[A(\lambda)-A(0)]v$, \eqref{eqn:A diff L2}, remains valid for $S_1v[A(\lambda)-A(0)]v$.  This yields \eqref{Gammakrep} when $n=4m$.

When $2m<n<4m$, first note that, using $v\mR_0v=\mathcal{AE}$, we have  
$vA(\lambda)v=\mathcal{AE}$, and hence $UvA(\lambda)v M_e(\lambda)vA(\lambda)vU=\mathcal {AE}$.

Second, by Lemma~\ref{lem:R0 for Mexp_small}, for some absolutely bounded operators $C_{\ell,\kappa}$, we have
$$
vA(\lambda)v=vA(0)v+\sum_{\ell=1}^{2m-2} \lambda^{\ell}vC_{\ell,\kappa}v + \lambda^{2m-1}\mathcal{AE}.  
$$
The $C_{\ell,\kappa}$ are either exactly zero, or may be explicitly constructed in terms of compositions of operators of the form $vG_j^1v$ and $vG_0^0v$, the exact form is unimportant for our purposes.
Finally, by \eqref{Al-A0} and Lemma~\ref{lem:one_side_cancel}, we have
\begin{multline*}
	S_1v[A(\lambda)-A(0)]v=\sum_{j=1}^{\kappa}(-1)^{j-1}S_1v [\mR_0(\lambda^{2m})-G_0^0] (V\mR_0(\lambda^{2m}))^{\kappa-j}v\\  =  \sum_{j=1}^{\kappa}(-1)^{j-1}S_1 [M_\ell(\lambda)+\lambda^{2m}\mathcal{AE}] (vU\mR_0(\lambda^{2m})v)^{\kappa-j} \\
	=\lambda^{2m}\mathcal{AE}+ \sum_{j=1}^{\kappa}(-1)^{j-1}S_1  M_\ell(\lambda)  (vU\mR_0(\lambda^{2m})v)^{\kappa-j} =[O_{\lceil\frac{n}{2}\rceil+1}(\lambda^{ \alpha_0}) +\lambda^{2m}] \mathcal{AE}
\end{multline*}
Combining all together, recalling that $\alpha_0\geq m+1$, we see that
$$
\lambda^{-2m}UvA(\lambda)v \widetilde M(\lambda)vA(\lambda)vU=\mathcal{AE}+\lambda^{-2m} \Gamma_s(\lambda),$$
where $\Gamma_s(\lambda)$ is a linear combination of operators of type $(k_1,k_2,\alpha)$ as $\widetilde M(\lambda)$ in Proposition~\ref{Minvsing}.

\end{document}